\documentclass[letterpaper, 12pt, journal, onecolumn, draftclsnofoot]{IEEEtran}
 
\usepackage{amsthm}
\usepackage{enumitem}

\usepackage{soul}
\usepackage[normalem]{ulem}
\setstcolor{red}
\usepackage{cancel}
\usepackage[latin1]{inputenc}
\usepackage[english]{babel}
\usepackage{amsfonts,amssymb,amsmath,comment}
\allowdisplaybreaks

\usepackage[dvipsnames]{xcolor}
\usepackage{tikz}
\usetikzlibrary{decorations, decorations.text,backgrounds}
\usepackage{graphicx}

\usepackage[hidelinks]{hyperref}
\usepackage{bookmark}

\usepackage{cite}
\usepackage{placeins}

\renewcommand{\natural}{{\mathbb{N}}}

\newcommand{\real}{{\mathbb{R}}}

\newcommand{\map}[3]{#1: #2 \rightarrow #3}

\newcommand{\norm}[1]{\|#1\|} 

\newcommand{\sign}{\operatorname{sign}}

\newcommand{\smallsum}{\textstyle\sum\limits}

\newcommand{\until}[1]{\{1,\ldots,#1\}}

\newcommand{\CC}{\mathcal{C}} 
\newcommand{\EE}{\mathcal{E}} 
\newcommand{\GG}{\mathcal{G}}

\newcommand{\KK}{\mathcal{K}}
\newcommand{\II}{\mathcal{I}} 
\newcommand{\NN}{\mathcal{N}} 
\renewcommand{\SS}{\mathcal{S}} 
 
\newcommand{\TT}{\mathcal{T}}

 \newcommand{\subj}{\text{subj.\ to}}

\newcommand{\argmin}{\mathop{\rm argmin}}

\newcommand{\nbrs}{\mathcal{N}}

\newcommand{\innbrs}{\mathcal{N}}

\usepackage{algorithm}
\usepackage[noend]{algpseudocode}

\makeatletter
\newcommand{\StatexIndent}[1][3]{%
  \setlength\@tempdima{\algorithmicindent}%
  \Statex\hskip\dimexpr#1\@tempdima\relax}
\makeatother

\algnewcommand{\algorithmicgoto}{\textbf{go to }}%
\algnewcommand{\Goto}[1]{\algorithmicgoto Line~\ref{#1}}%
\algnewcommand{\Label}{\State\unskip}

\newtheorem{theorem}{Theorem}[section]
\newtheorem{proposition}[theorem]{Proposition}

 \newtheorem{lemma}[theorem]{Lemma}
\newtheorem{remark}[theorem]{Remark}

\newtheorem{assumption}[theorem]{Assumption}

\newcommand{\GS}[1]{{\color{red} #1}}

\renewcommand{\inf}{\operatornamewithlimits{inf\vphantom{p}}}
\renewcommand{\liminf}{\operatornamewithlimits{liminf\vphantom{p}}}
\renewcommand{\limsup}{\operatornamewithlimits{limsup\vphantom{p}}}
\renewcommand{\lim}{\operatornamewithlimits{lim\vphantom{p}}}

\newcommand\oprocendsymbol{\hbox{$\square$}}
\newcommand\oprocend{\relax\ifmmode\else\unskip\hfill\fi\oprocendsymbol}
\def\eqoprocend{\tag*{$\square$}}

\graphicspath{{figs/}}

\newcommand{\bx}{\mathbf{x}}
\newcommand{\bs}{\mathbf{s}}
\newcommand{\bsigma}{\boldsymbol{\sigma}}
\newcommand{\tildex}{\widetilde{\mathbf{x}}}
\newcommand{\hatx}{\widehat{\mathbf{x}}}

\newcommand{\Deltax}{\Delta \mathbf{x}}

\newcommand{\bDelta}{\boldsymbol{\Delta}}

\newcommand{\avgdec}{\bar{\boldsymbol{s}}}
\newcommand{\avggrad}{\bar{\boldsymbol{\sigma}}}
\newcommand{\hath}{\widehat{h}}

\newcommand{\hbu}{\widehat{\mathbf{u}}}
\newcommand{\hbg}{\widehat{\mathbf{g}}}
\newcommand{\bv}{\mathbf{v}}
\newcommand{\bu}{\mathbf{u}}
\newcommand{\bn}{\mathbf{n}}
\newcommand{\bw}{\mathbf{w}}
\newcommand{\by}{\mathbf{y}}
\newcommand{\bz}{\mathbf{z}}
\newcommand{\bA}{\mathbf{A}}
\newcommand{\bD}{\mathbf{D}}
\newcommand{\bb}{\mathbf{b}}

\newcommand{\1}{\mathbf{1}}
\newcommand{\0}{\mathbf{0}}

\newcommand{\bepsilon}{\boldsymbol{\epsilon}}
\newcommand{\boldeta}{\boldsymbol{\eta}}

\newcommand{\kron}{\otimes}
\newcommand{\tf}{\tilde{f}}

\newcommand{\surr}{\widehat{f}}
\newcommand{\reg}{r}
\newcommand{\subgrad}{\widetilde{\nabla}}

\definecolor{blue@O4S}{RGB}{0, 41, 69}
\definecolor{emph@O4S}{RGB}{0, 93, 137}
\definecolor{red@O4S}{RGB}{127,0,0}
\definecolor{gray@O4S}{RGB}{112, 112, 112}

\def \algname/{{\sc Block-SONATA}}

\def \blockpushsum/{{Block-wise Push-sum Average Consensus}}

\begin{document}

\title{
Distributed Big-Data Optimization \\via Block-wise Gradient Tracking
}

\author{Ivano Notarnicola$^{\ast}$, Ying Sun$^{\ast}$, Gesualdo Scutari, 
Giuseppe Notarstefano
\thanks{$^\ast$These authors equally contributed and are in alphabetic order.}
  \thanks{
  Preliminary short versions of this paper have appeared 
  as~\cite{notarnicola2017cdc,notarnicola2017camsap}.
  }

\thanks{The work of Notarnicola and Notarstefano has received funding from the
  European Research Council (ERC) under the European Union's
  Horizon 2020 research and innovation programme (grant agreement No 638992 -
  OPT4SMART).
  The work of Sun and Scutari has been supported by the USA National Science
  Foundation under Grants CIF 1564044, CIF 1719205,  and CAREER Award 1555850; and in part 
  by the Office of Naval Research under the Grant N00014-16-1-2244 and  the Army Research
Office under Grant W911NF1810238.
  }
\thanks{Ivano Notarnicola is with the
  Department of Engineering, Universit\`a del Salento, Lecce, Italy, 
  \texttt{ivano.notarnicola@unisalento.it.}  } 
\thanks{Ying Sun and Gesualdo Scutari are with the School of Industrial Engineering, 
Purdue University, West-Lafayette, IN, USA, \texttt{\{sun578,gscutari\}@purdue.edu.}}
\thanks{Giuseppe Notarstefano is with the Department of Electrical, Electronic and Information 
Engineering,
University of Bologna, Bologna, Italy, 
\texttt{giuseppe.notarstefano@unibo.it}}
}

\maketitle

\begin{abstract}
  We study \emph{distributed big-data nonconvex} optimization in multi-agent
  networks.  We consider the (constrained) minimization of the sum of a smooth
  (possibly) nonconvex function, i.e., the agents' sum-utility, plus a convex
  (possibly) nonsmooth regularizer. Our interest is on big-data problems in
  which there is a large number of variables to optimize. If treated by means of
  standard distributed optimization algorithms, these large-scale problems may
  be intractable due to the prohibitive local computation and communication
  burden at each node.  We propose a novel distributed solution method where, at
  each iteration, agents update in an uncoordinated fashion only one block of the
  entire decision vector. To deal with the nonconvexity of the cost function,
  the novel scheme hinges on Successive Convex Approximation (SCA) techniques
  combined with a novel \emph{block-wise} perturbed push-sum consensus
  protocol, which is instrumental to perform local block-averaging operations
  and tracking of gradient averages. Asymptotic convergence to stationary
  solutions of the nonconvex problem is established. Finally, numerical results
  show the effectiveness of the proposed algorithm and highlight how the block
  dimension impacts on the communication overhead and practical convergence
  speed.
\end{abstract}

\section{Introduction}
\label{sec:intro}
Many modern control, estimation and learning applications lead to large-scale  optimization
problems, i.e., problems with a huge number of variables to optimize. %
These problems
are often referred to as \emph{big-data}, and call for the design of  tailored
algorithms.
In this paper we consider  \emph{distributed} (nonconvex) \emph{big-data
optimization}. That is, we aim at solving large-scale optimization  problems over networks in a
distributed way by addressing the following two challenges: (i) \emph{optimizing over
(or even computing the gradient with respect to) all the variables can be too
costly, and} (ii) \emph{broadcasting to neighbors the entire solution estimate would
incur in an unaffordable communication overhead.}
The literature on parallel and distributed methods is abundant; however,  we are not aware of any  work that can deal with both challenges (i) and (ii) over networks, as detailed next. 

\subsection{Related Works}
We organize the relevant literature in two main groups: centralized and parallel
algorithms for large-scale optimization; and distributed algorithms
applicable to multi-agent networks (with no specific topology).

\noindent \textbf{Parallel algorithms.} %
Parallel Block-Coordinate-Descent (BCD) methods are well-established methods in optimization;
more recently, they have been proven to be  particularly effective in solving
very large-scale (mainly convex)  optimization problems arising, e.g.,  from  data-intensive applications. 
Examples include %
\cite{nesterov2012efficiency} for convex, smooth functions, and 
\cite{richtarik2012parallel,necoara2016parallel}
for composite optimization;  a detailed overview of BCD methods can be found in
\cite{Wright2015}.
Parallel solution methods based on  Successive Convex Approximation (SCA) techniques 
have been proposed in~\cite{facchinei2015parallel} to deal with nonconvex problems; 
see~\cite{Multi-agent-book} for a recent research tutorial on the subject.
In~\mbox{\cite{mokhtari2016doubly}} block coordinate-descent and stochastic-gradient
methods have been combined to optimize big-data, sum-of-utilities (cost)
functions.
These algorithms, however, are not implementable in a (fully) distributed
setting; they are instead designed to be run on ad-hoc computational
architectures, e.g., shared-memory systems or star networks.

\noindent \textbf{Distributed multi-agent algorithms.} %
The literature on distributed methods for multi-agent optimization is vast. Here, we discuss  only primal-based algorithms, 
as they are more closely related to the approach  proposed in this paper.
Distributed subgradient methods have been proposed in the early works
\cite{nedic2009distributed,nedic2010constrained}, to solve convex, problems over
undirected graphs. The extension to nonconvex
costs has been developed in~\cite{bianchi2013convergence}.
The generalization to (time-varying) digraphs was studied
in~\cite{nedic2015distributed} and \cite{tatarenko2017nonconvex} for convex and
nonconvex objectives, respectively; these schemes combine distributed
(sub-)gradient with push-sum consensus \cite{benezit2010weighted} updates.
A Nesterov acceleration of the mentioned approach applied to convex,
smooth problems has been proposed in~\cite{jakovetic2014fast} with a convergence
rate analysis.
Local, private constraints are handled in \cite{lee2016asynchronous} and
\cite{margellos2018distributed}, where distributed methods based on a random
projection subgradient and a proximal minimization are proposed respectively.
All these methods need to use a diminishing step-size to converge to an exact,
consensual solution, thus converging at a sub-linear rate.  On the other hand,
with a constant (sufficiently small) step-size, they can be faster, but they
would converge only to a neighborhood of the solution set.

Primal-based distributed methods that converge to an exact consensual
solution using fixed step-sizes are available in the literature; they can 
be roughly grouped as i)
\cite{shi2015extra,shi2015proximal}; ii)
\cite{zanella2011newton,zanella2012asynchronous,varagnolo2016newton}, iii)
\cite{dilorenzo2015distributed,dilorenzo2016next,xu2015augmented,sun2016distributed,sun2017distributed,Multi-agent-book,xu2018convergence};
and iv) \cite{nedic2017achieving,qu2017harnessing,xi2018addopt,xin2018linear}. 
While substantially different, these schemes build on the idea of correcting the
decentralized gradient- (or Newton-) related direction to cancel the steady state
error in it. More specifically, in \cite{shi2015extra} and its proximal variant
\cite{shi2015proximal}, the gradient direction is corrected using iterate and
gradient information of the last two iterations. 
In~\cite{zanella2011newton,zanella2012asynchronous,varagnolo2016newton}, 
the novel idea of distributively estimating a Newton-Raphson direction by 
means of suitable average consensus ratios has been
introduced. In~\cite{carli2015analysis} the same approach has been extended 
to deal with directed, asynchronous networks with lossy communications.
The third and fourth class of works is based on the idea of gradient tracking:
each agent updates its own local variables along a surrogate direction that
tracks the gradient of the sum-utility (which is not locally available). This
idea was proposed independently in
\cite{dilorenzo2015distributed,dilorenzo2016next} for constrained nonsmooth
nonconvex problems, and in \cite{xu2015augmented,xu2018convergence} for strongly convex,
unconstrained, smooth, optimization. The works
\cite{sun2016distributed,sun2017distributed,Multi-agent-book} extended the algorithms to
(possibly) time-varying digraphs (still in the nonconvex setting of
\cite{dilorenzo2015distributed,dilorenzo2016next}). A convergence rate analysis
of the scheme \cite{xu2015augmented} was later developed in
\cite{nedic2016cdc,nedic2017achieving,qu2016cdc,qu2017harnessing}, 
with~\cite{nedic2016cdc,nedic2017achieving} considering time-varying (directed) graphs. 
Another scheme, still based on the
idea of gradient tracking, has been recently proposed in \cite{xin2018linear}.
All the above methods are based on the optimization and communication at each 
iteration of the \emph{entire} set of variables of every agents (or some related 
quantities of the same size).

First attempts to block-wise distributed optimization have been proposed
in~\cite{carli2013distributed,notarnicola2016randomized,notarnicola2018partitioned}
for a structured, \emph{partitioned} optimization set-up in which the cost
function of each agent depends on its (block) variables and those of its
neighbors. In~\cite{wang2018coordinate} a distributed stochastic gradient
method has been proposed whereby agents optimize at each iteration only 
a subset of their variables (still communicating the entire vector).

\subsection{Major Contributions} 
We propose a distributed algorithm over networks for, possibly nonconvex,
big-data optimization problems, that explicitly accounts for challenges (i) and
(ii).
To cope with these two challenges, we propose a distributed scheme in which, at
every iteration, each agent optimizes over and communicates only \emph{one
  block} of the local solution estimate (and of auxiliary vectors) rather than
all the components.  Blocks are selected in an uncoordinated fashion by means of
an ``essentially cyclic rule'', thus guaranteeing all of them to be persistently
updated during the algorithmic evolution.
Specifically, inspired to the two optimization algorithms 
NEXT (in-Network succEssive conveX approximaTion)
\cite{dilorenzo2015distributed,dilorenzo2016next} and SONATA 
(distributed Successive cONvex Approximation algorithm over 
Time-varying digrAphs) \cite{sun2016distributed,sun2017distributed},
\emph{not suitable for big-data}
problems, we propose a block-iterative two-step (optimization and averaging)
procedure, named \algname/.
Each agent solves a (small) local optimization problem, depending only on the
selected block, with cost function being a strongly convex surrogate of the
nonconvex sum-cost function, whose gradient is a local estimate of the total
gradient of the (smooth part of the) sum-cost function.
The (block-wise) optimization step is combined with a twofold \emph{block-wise}
perturbed averaging scheme on the local solution estimate and on the local
estimate of the total gradient.
This scheme guarantees both the asymptotic agreement of the local
solution estimates and the tracking of total gradient.
We remark that this novel block-wise perturbed averaging protocol
  extends a (static) block averaging protocol proposed for an abstract message
  passing model in \cite{tsitsiklis1986distributed}, and is thus of independent
  interest for other distributed computation tasks. It can be used by agents of
  a network to reach consensus or track the average of local signals by
  exchanging with neighboring agents only one block of their local vector.
For the proposed distributed optimization algorithm we prove that: local
solution estimates are asymptotically consensual to their (weighted) average,
and any limit point of the average sequence is a stationary solution of the
optimization problem.
The algorithm analysis has two key distinctive features. First, a proper
convergence analysis of the block-wise perturbed averaging scheme is developed
based on suitable block-induced time-varying digraphs. Second, errors due to
inexact block-wise minimizations and to uncoordinated block updates are properly
handled to show that a suitably designed merit function decreases along the 
algorithmic evolution.

The rest of the paper is organized as follows. In Section~\ref{sec:setup} we present
the problem set-up. In Section~\ref{sec:block_consensus} we introduce a
block-wise perturbed consensus scheme that will act as a building block for our
distributed big-data optimization algorithm presented in
Section~\ref{sec:algorithm}, along with its convergence properties. In
Section~\ref{sec:simulations} we provide numerical computations to test our
algorithm. Finally, the convergence analysis is deferred to the appendix.

\section{Distributed Big-Data Optimization Set-up}
\label{sec:setup}

We consider a multi-agent system composed of $N$ agents, aiming at cooperatively
solving the following composite (possibly) nonconvex large-scale optimization problem
\begin{align}
\begin{split}
  \min_{\bx} \hspace{0.3cm} & \: 
  U ( \bx) \triangleq \sum_{i=1}^N f_i ( \bx ) + \sum_{\ell = 1}^B \reg_\ell ( \bx_\ell )
  \\
  \subj \: & \: \bx_\ell \in \KK_\ell, \quad \ell \in\until{B},
\end{split}
\label{eq:problem}
\end{align}
where  $\bx$ is the vector of optimization variables, partitioned in $B$ blocks as
\begin{align*}
  \bx \triangleq 
  \begin{bmatrix}
    \bx_1 
    \\[-0.3em] 
    \vdots 
    \\[-0.3em] 
    \bx_B
  \end{bmatrix},
\end{align*}
with $\bx_\ell\in\real^{d}$, $\ell\in\until{B}$; $\map{f_i}{\real^{dB}}{\real}$ 
is the cost function of agent $i$, assumed to be smooth but (possibly) nonconvex; 
$\map{\reg_\ell}{\real^{d}}{\real}$, $\ell\in\until{B}$, is 
a convex (possibly) nonsmooth function; and $\cal K_\ell$, $\ell\in\until{B}$, is a closed 
convex set; we denote by $\KK\triangleq \KK_1\times \cdots \times \KK_{B}$ 
the feasible set of~\eqref{eq:problem}.
The nonsmooth terms $\reg_\ell$ are usually used to promote
some extra structure in the solution, such as (group) sparsity.
We study~\eqref{eq:problem} under the following assumption.

\begin{assumption}[On the Optimization Problem]\mbox{}
\label{ass:cost_functions}
\begin{enumerate}[leftmargin=0.65cm]
  \item Each $\KK_\ell \neq \emptyset$  is closed and convex;
  \item Each $\map{f_i}{\real^{dB}}{\real}$ is 
  $\CC^1$ on (an open set containing) $\KK$;
  \item Each $\nabla f_i$ is $L_i$-Lipschitz  continuous 
    and bounded on $\KK$;
  \item Each $\map{\reg_\ell}{\real^d}{\real}$ is convex (possibly nonsmooth)  on $\KK$,
  with bounded subgradients on $\KK$;%
  \item $U$ is coercive on $\KK$, i.e., 
  $\lim_{\bx\in\KK,\|\bx\|\to\infty} U(\bx) = \infty$. \oprocend
\end{enumerate}
\end{assumption}
The above assumptions are quite standard and satisfied by many practical  problems; 
see, e.g.,~\cite{facchinei2015parallel}. Here, we only remark that we do not assume any 
convexity of $f_i$. In the following, we also make the blanket assumption that each agent 
$i$ knows only its own cost function $f_i$, the regularizers $\reg_\ell$ and the feasible 
set $\KK$, but not the other agents' functions.

\smallskip
\noindent \emph{On the communication network:} 
The  communication among agents is modeled as a fixed, directed
graph $\GG = (\until{N},\EE)$, where $\EE\subseteq \until{N} \times \until{N}$
is the set of edges. The edge $(i,j)\in \EE$ models the fact that agent $i$
can send a message to agent $j$.
We denote by $\innbrs_i$ the set of \emph{in-neighbors} of node $i$ in the fixed
graph $\GG$, i.e.,
$\innbrs_i \triangleq \{j \in \until{N} \mid (j,i) \in \EE \}$.  We
assume that $\EE$ contains self-loops and, thus, $\innbrs_i$ contains $\{i\}$
itself.
We use the following assumption.
\smallskip
\begin{assumption}
The digraph $\GG$ is strongly connected.~\oprocend
\label{ass:strong_conn}
\end{assumption} 
\smallskip

\noindent\emph{Algorithmic Desiderata:} Our goal is to solve problem~\eqref{eq:problem} in 
a distributed fashion, leveraging local communications among neighboring agents. 
As a major departure from current  literature on distributed optimization,  
here we focus on \emph{big-data} instances of~\eqref{eq:problem} 
in which the vector of variables $\bx$ is composed by a huge number of 
components ($B$ is very large).  In such problems, 
minimizing the sum-utility with respect to all the components of $\bx$, or even computing the 
gradient or evaluating the value of a single function $f_i$, can require substantial computational efforts. 
Moreover, exchanging  an estimate of the \emph{entire} local decision variables 
$\bx$ over the network (like current distributed schemes do) is not efficient 
or even feasible, due to the excessive communication overhead. We design next 
the first scheme able to deal with such challenges.

\section{Block-wise Perturbed Push-Sum Consensus}
\label{sec:block_consensus}

In this section we design a building block of our distributed optimization
algorithm, namely a \emph{block-wise perturbed push-sum consensus algorithm}. 
We first devise  the ``unperturbed'' instance of the scheme, suitable to solve the average 
consensus problem over digraphs via block-communications.
Then, we introduce the general perturbed version of the scheme, which allows
agents to solve more general tasks, such as tracking the average of given
(time-varying) agents' signals.

\subsection{Block-wise Push-sum Average Consensus}
Consider a system of $N$ agents aiming at reaching consensus on
the average of given initial values. Let the communication network be
modeled as a digraph $\GG$ satisfying Assumption~\ref{ass:strong_conn}. 
To solve this problem, one can leverage the popular push-sum (consensus)
algorithm \cite{benezit2010weighted}. 
However, differently from this scheme in which agents need to exchange their \emph{entire}
local estimates at each iteration, here we consider a \emph{block-wise}
communication protocol. 
Specifically, while at every iteration $t$ each agent $i$ can update its entire
(average estimate) vector $\bz_{(i,:)}^t \in \real^{dB}$, 
it sends to out-neighbors \emph{one block only}. Let $\bz_{(i,\ell_i^t)}^t\in \real^{d}$
denote the $\ell_i^t$-th block that, at time $t$, agent $i$ selects (according to a proper rule) 
and broadcasts to its out-neighbors. To update  $\bz_{(i,:)}^t$, agent $i$ runs a push-sum consensus  on each
block $\ell$ of $\bz_{(i,:)}^t$ {\it separately}, using only the information received from its in-neighbors that sent
block $\ell$ at time $t$ (if any).

Since no coordination is assumed among agents in selecting their blocks, 
different agents will likely select blocks with different index, i.e., $\ell_i^t \neq \ell_j^t$, with $i \neq j$.
This induces a \emph{block-dependent} communication graph, one for each block index $\ell$, which 
is, in general, a subgraph of $\GG$. In this subgraph, agent $j$ is an in-neighbor of agent $i$ at time $t$ 
if $j\in \nbrs_i$ and $\ell_j^t = \ell$, i.e., agent $j$ sent its block $\ell$ to $i$ at time $t$.
This suggests the definition of \emph{block-dependent} neighbor sets. For each agent $i\in\until{N}$ and 
each block $\ell\in\until{B}$, define
\begin{align*}
  \nbrs_{i,\ell}^t \triangleq \{ j \in \nbrs_i \mid \ell_j^t = \ell \}  \cup \{i \} \subseteq \nbrs_i,
\end{align*}
which includes, besides agent $i$, only the in-neighbors of agent $i$ in $\GG$ that
sent (i.e., updated) block $\ell$ at time $t$.
Consistently, we denote by $\GG_\ell^t\triangleq (\until{N},\EE_\ell^t)$ the \emph{time-varying} subgraph of
$\GG$ associated to block $\ell$ at iteration $t$. Its edge set is $\EE_\ell^t
\triangleq \{(j,i)\in\EE \mid j\in \nbrs_{i,\ell}^t, i\in\until{N}\}.$

Note that each (time-varying) digraph $\GG_\ell^t$ is induced by the block selection 
rules (independently) adopted by the agents, so that the connectivity properties
of all digraphs are coupled; this interplay will be discussed shortly (cf. Assumption~\ref{ass:block_selection} and 
Proposition~\ref{prop:induced_graph_connectivity}).

The following table ``\blockpushsum/'' formally introduces the algorithm 
from the perspective of agent $i$ only.
The algorithm consists of applying the push-sum consensus 
protocol in a \emph{block-wise} fashion over the time-varying subgraphs 
$\GG_\ell^t$ introduced above. 
As in the existing consensus protocols, $a_{ij\ell}^t$ in \eqref{eq:block_consensus}
are nonnegative weights to be properly chosen. We let $\bA_\ell^t \triangleq
[a_{ij\ell}^t]_{i,j=1}^N$ be the weight-matrix matching $\GG_\ell^t$
(cf. Assumption~\ref{ass:col_stoch}).
Each agent $i \in \until{N}$ initializes its local variables as $\phi_{(i,\ell)}^0 = 1$ 
and $\bz_{(i,\ell)}^0$ an arbitrary value in $\real^d$ for all $\ell\in\until{B}$.
\renewcommand{\thealgorithm}{}
\floatname{algorithm}{}

\begin{algorithm}[!ht]
  \begin{algorithmic}
      \smallskip

    \StatexIndent[0.2] Select $\ell_i^t \in \until{B}$

    \StatexIndent[0.2]
    For each $j \in \innbrs_i$ receive $\phi_{(j,\ell_j^t)}^{t}$ and 
    $\bz_{(j,\ell_j^t)}^{t}$
    \StatexIndent[0.2] For each $\ell \in \until{B}$ compute
    \begin{align}\label{eq:block_consensus}
    \begin{split}
      \phi_{(i,\ell)}^{t+1} 
      & = \sum_{j\in \innbrs_{i,\ell}^t } a_{ij\ell}^t \, \phi_{(j,\ell)}^{t}
      \\
      \bz_{(i,\ell)}^{t+1} & = 
      \sum_{j\in \innbrs_{i,\ell}^t } 
      \frac{a_{ij\ell}^t \,\phi_{(j,\ell)}^{t} }{\phi_{(i,\ell)}^{t+1}}
        \, \bz_{(j,\ell)}^{t} 
    \end{split}
    \end{align}
       
  \end{algorithmic}
  \caption{\blockpushsum/}
  \label{alg:block-wise_push-sum}
\end{algorithm}

Convergence of the \blockpushsum/
depends on the choice of the weight matrices as well as the block-selection rules 
employed by the agents (which affect the connectivity properties of each digraph 
sequence $\{\GG_\ell^t\}_{t\ge 0}$, $\ell\in\until{B}$).  
Sufficient conditions on these parameters guaranteeing convergence are 
discussed next.\smallskip

\noindent \emph{On the choice of $\bA_\ell^t$}: We make the following assumption 
on each $\bA_\ell^t$, which is standard for the push-sum algorithm.
\smallskip 
\begin{assumption}
  Given the sequence of graphs $\{\GG_\ell^t\}_{t\ge 0}$, $\ell\in\until{B}$ and $t\geq 0$, 
  each matrix $\bA_\ell^t$ satisfies the following: 
  \begin{enumerate}
  	\item[(a)] $a_{ij\ell}^t > \kappa$, if $(j,i)\in\EE_\ell^t$; and $a_{ij\ell}^t =0$, if $(j,i)\notin\EE_\ell^t$;
  	\item[(b)] it is column stochastic, 
  that is, $\1^\top \bA_\ell^t = \1^\top$;
  \end{enumerate}
  where $\kappa$ is some positive constant.~\oprocend
  \label{ass:col_stoch}
\end{assumption}

A natural question is whether a matrix $\bA_\ell^t$ satisfying 
Assumption~\ref{ass:col_stoch} can be build by the agents using 
only local information. Next, we propose a simple procedure to locally build a
valid $\bA_\ell^t$.
Being the underlying communication digraph $\GG$ static and 
strongly connected (cf.~Assumption~\ref{ass:strong_conn}), we assume
that a column stochastic matrix $\tilde{\bA}$ matching $\GG$ is available,
i.e., $\tilde{a}_{ij} > 0$ if $(j,i)\in\EE$ and $\tilde{a}_{ij} = 0$
otherwise; and $\1^\top \tilde{\bA} = \1^\top$.
To construct $\bA_\ell^t$  in a distributed way,
we start noticing that at iteration $t$, an agent $j$ either sends a block $\ell$ to all
its out-neighbors in $\GG$, $\ell=\ell_j^t$, or to none, $\ell\neq\ell_j^t$.
Thus, let us focus on the $j$-th column of $\bA_\ell^t$, denoted 
by $\bA_\ell^t(:,j)$. If agent $j$ does not send
block $\ell$ at iteration $t$, $\ell\neq\ell_j^t$, all the components of
$\bA_\ell^t(:,j)$ will be zero except for $a_{jj\ell}^t$. Thus, for the  $j$-th column to 
be stochastic, it must be $\1^\top  \bA_\ell^t(:,j) = a_{jj\ell}^t=1$ (i.e.,
$\bA_\ell^t(:,j)$ is the $j$-th vector of the canonical basis). 
Vice versa, if $j$ sends block $\ell$, all its out-neighbors in $\GG$ will
receive it and, thus, column $\bA_\ell^t(:,j)$ has the same nonzero entries as
column $\tilde{\bA}(:,j)$ of $\tilde{\bA}$. 
Since $\tilde{\bA}$ is column stochastic, one can set $\bA_\ell^t(:,j)=\tilde{\bA}(:j)$.
Note that each agent can locally construct its own weights satisfying 
the above rule. 
In summary, for each $i\in\until{N}$ and $\ell\in\until{B}$, weights 
$a_{ij\ell}^t$ can be chosen as
\begin{align}
  a_{ij\ell}^t \triangleq 
  \begin{cases}
	  \tilde{a}_{ij}, \: \: & \text{if } j \in \nbrs_i \text{ and } \ell = \ell_j^t,
	  \\
	  1,                    & \text{if } j=i \text{ and } \ell \neq \ell_i^t, 
	  \\
	  0,                    & \text{otherwise}.
	\end{cases}
	\label{eq:weights}
\end{align}

\noindent \emph{On the choice of  the block selection rule}: To guarantee convergence 
of the \blockpushsum/ over time-varying digraphs, it is well known that some long-term 
connectivity property is required on the digraph sequence \cite{benezit2010weighted}.
Here, we use $T$-strong connectivity: for each $\ell\in \until{B}$, the time-varying digraphs
$\{\GG_\ell^t\}$ are $T$-strongly connected, i.e., the union 
digraph $\bigcup_{\tau=0}^{ T-1} \GG_\ell^{t+\tau}$ is strongly connected $\forall \, t \ge 0$.

The $T$-strong connectivity requirement imposes a condition on the way the 
blocks are selected.
Note that $\GG_\ell^t$ is a subgraph of $\GG$ such that
if agent $i$ selects (sends) block $\ell$ at time $t$, then 
the edges in $\EE$ leaving node $i$ are also present in $\EE_\ell^t$.
Hence, since $\GG$ is strongly connected (cf.~Assumption~\ref{ass:strong_conn}),
 the following general \emph{essentially cyclic} rule is enough 
to guarantee that each $\{\GG_\ell^t\}$ is $T$-strongly connected.\smallskip 

\begin{assumption}[Block Selection Rule]
\label{ass:block_selection}
	For each	agent $i\in\until{N}$ there exists a (finite) constant $T_i>0$ such that
	\begin{equation*} 
	  \bigcup_{\tau=0}^{T_i-1} \{\ell_i^{t+\tau} \} = \until{B}, \text{ for all } t \ge 0. 
	  \eqoprocend
  \end{equation*}
\end{assumption}
\smallskip 

Note  that the above rule does not impose any coordination among the agents; they  select their own block independently. Therefore,   at a given iteration $t$, 
different agents may update different blocks. Moreover, some blocks can be updated 
more often than others.
On the other hand, such a rule guarantees that, within a finite time window of length $T \le \max_{i\in\until{N}} T_i$,
all the blocks have been updated at least once by all agents. 
This is sufficient to ensure that $\GG_\ell^t$ is $T$-strongly connected, as 
formally stated next. 

\begin{proposition}
  Under Assumption~\ref{ass:strong_conn} and \ref{ass:block_selection}, there
  exits a $0<T\leq \max_{i\in\until{N}} T_i$, such that, for each
  $\ell\in \until{B}$, the union digraph
  $\bigcup_{\tau = 0}^{ T-1} \GG_\ell^{t+\tau}$ is strongly connected, for all
  $t\geq 0$.
\label{prop:induced_graph_connectivity}
\end{proposition}
\begin{proof}
Consider a block $\ell$ and define $t+s_i^t(\ell)$ as the last iteration 
in which agent $i$ sends block $\ell$ within the time window $[t,t+T-1]$. 
The essentially cyclic rule (cf.~Assumption~\ref{ass:block_selection}) 
implies that $0\le s_i^t(\ell) \leq T-1 $ for all $i\in \until{N}$. 
By definition of $\GG_\ell^t$, we have that any edge $(j,i) \in \EE$ also belongs 
to $G_\ell^{t+ s_i^t(\ell)}$. 
Since 
$\EE \subseteq \bigcup_{i=1}^N \GG_\ell^{t+ s_i^t(\ell)} \subseteq 
\bigcup_{\tau=0}^{T-1} \GG_\ell^{t+\tau}$, we have 
$\bigcup_{\tau = 0}^{ T-1} \GG_\ell^{t+\tau}$ is strongly connected because also 
$\GG$ is so (cf.~Assumption~\ref{ass:strong_conn}).
\end{proof}

\subsection{Block-wise Perturbed Push-sum}

We can now generalize the  \blockpushsum/ introducing in the agents' local 
updates a local block-wise, time-varying perturbation, denoted 
by $ \bepsilon_{(j,\ell)}^t$. 
The block-wise perturbed push-sum can be obtained by replacing the
update~\eqref{eq:block_consensus} with the following perturbed version 
\begin{align}
\begin{split}
	 \phi_{(i,\ell)}^{t+1}  & = \sum_{j\in \innbrs_{i,\ell}^t }
   a_{ij\ell}^t \,\phi_{(j,\ell)}^t,
   \\
   \bz_{(i,\ell)}^{t+1} 
   & = \sum_{j\in \innbrs_{i,\ell}^t } 
   \dfrac{a_{ij\ell}^t \, \phi_{(j,\ell)}^{t} }{\phi_{(i,\ell)}^{t+1}} 
   \, 
   \big( 
   \bz_{(j,\ell)}^t
   +
   \bepsilon_{(j,\ell)}^t
   \big),
\end{split}
\label{eq:perturbed_block_consensus}
\end{align}
for all $\ell \in \until{B}$, where each $\bepsilon_{(i,\ell)}^t \in \real^d$ is a 
suitable perturbation that each agent injects in its update. This scheme is a 
building block of proposed block-wise distributed optimization algorithm that 
will be introduced in the next section.
Convergence of the block-wise perturbed push-sum algorithm is stated in the 
following proposition.

\begin{proposition}
\label{prop:perturbed_push-sum_convergence}
  Consider the block-wise perturbed push-sum 
  consensus~\eqref{eq:perturbed_block_consensus}, with weight matrix $\bA_\ell^t$ defined
  according to~\eqref{eq:weights}. Then, under
  Assumptions~\ref{ass:strong_conn} and~\ref{ass:block_selection},
 there holds
  \begin{align*}
  & 
  \Big \| \bz_{(i,:)}^{t} - \dfrac{1}{N} \sum_{j=1}^N
  \bz_{(j,:)}^t \Big \|_1 
  =
  \sum_{\ell=1}^B
  \Big \| \bz_{(i,\ell)}^{t} - \dfrac{1}{N} \sum_{j=1}^N
  \bz_{(j,\ell)}^t \Big \|_1
  \\
  & \hspace{2.3cm}
  \le 
  c_1 (\rho)^t 
  +
  c_2 
  \sum_{\ell=1}^B
  \sum_{\tau =1}^{t}
  (\rho)^{t-\tau} 
  \sum_{j=1}^N 
  \| \bepsilon_{(j,\ell)}^\tau \|_1,
  \end{align*}
  with $\rho \in (0,1)$, for all $i\in\until{N}$. \oprocend
\end{proposition}
The proof of the proposition can be obtained by the proof 
of \cite[Lemma~1]{nedic2015distributed}, which we report in 
Appendix as Lemma~\ref{lem:perturbed_pushsum}, 
in vector form, as a preliminary result needed for our analysis.
As a corollary (with no proof), the previous result states that if the
perturbations $\bepsilon_{(i,\ell)}^t$ are vanishing, i.e,
$\lim_{t\to\infty} \|\bepsilon_{(i,\ell)}^t\| = 0$, for all $\ell \in \until{B}$
and $i\in \until{N}$, it holds
$\lim_{t\to\infty} \Big \| \bz_{(i,:)}^{t} - \frac{1}{N} \sum_{j=1}^N
\bz_{(j,:)}^t \Big \| = 0$,
for all $i\in\until{N}$. Clearly, for $\bepsilon_{(i,\ell)}^t = \0$ for all
$t\ge 0$, $\ell\in \until{B}$ and $i\in\until{N}$, the block-wise perturbed
push-sum reduces to the \blockpushsum/.

Several tasks can be accomplished by suitably choosing the perturbation
$\bepsilon_{(i,\ell)}^t$ in~\eqref{eq:perturbed_block_consensus}.
As a case study, in the following we show how to choose the perturbation,
in a block-wise fashion, in order to track the average of time-varying signals 
over graphs. 
The resulting block-wise tracking scheme will be part of the proposed distributed 
optimization algorithm.
\smallskip 

\noindent\textbf{Block-wise average signal tracking.} 
Consider the problem of tracking the average of $N$ time-varying signals over a
graph $\GG$, \cite{zhu2010discrete,kia2018tutorial}.
Specifically, assume each agent $i$ can generate (or evaluate) a time-varying 
signal, say 
$\{\bu_i^t\}_{t\in \natural}$, with each $\bu_i^t\in \real^{dB}$,
and aims at tracking the average signal $\bar{\bu}^t \triangleq (1/N) \cdot \sum_{i=1}^N \bu_i^t$ 
by exchanging information over the network.
Existing tracking schemes, e.g. ones used in distributed optimization algorithms~
\cite{dilorenzo2015distributed,xu2015augmented,varagnolo2016newton,
dilorenzo2016next,qu2017harnessing,qu2016cdc,sun2016distributed,xin2018linear,
xi2018addopt,
sun2017distributed,nedic2016cdc,nedic2017achieving,xu2018convergence},
require the acquisition and communication at each iteration of the entire signal
$\bu_i^t$, which might be too costly in a big-data setting. To cope with the curse of 
dimensionality, we can leverage the block-wise perturbed push-sum consensus
algorithm: to track distributedly $\bar{\bu}^t$, one can show that
it is sufficient to set $\bepsilon_{(i,\ell)}^t$ in~\eqref{eq:perturbed_block_consensus} to
\begin{align}
\label{eq:perturbation_for_tracking}
  \bepsilon_{(i,\ell)}^t =
  \frac{1}{\phi_{(i,\ell)}^t}  \Big ( \bu_{i,\ell}^{t+1} - \bu_{i,\ell}^t \Big ),
\end{align}
where $\bu_{i,\ell}^t$ denotes the $\ell$-th block of $\bu_{i}^t$.

While the tracking scheme \eqref{eq:perturbed_block_consensus}--\eqref{eq:perturbation_for_tracking} 
unlocks block-communications over networks, it requires, at each iteration, to potentially
perform~\eqref{eq:perturbation_for_tracking} for all the blocks $\ell\in\until{B}$, i.e. the 
evaluation (acquisition) of the \emph{entire} signal $\bu_i^t$.
When the cost of acquiring  $\bu_i^t$ is non-negligible, e.g., 
$\bu_i^t$ can be the gradient of a function with respect to a large number of 
variables, it is advisable to modify the protocol so that, at each iteration, only \emph{one} 
block of $\bu_i^t$ is used.   
To this end, we propose to replace $\bu_i^t$ with a surrogate local variable, 
denoted by $\hbu_i^t$, initialized as $\hbu_i^0 = \bu_i^0$. At each iteration $t$, 
agent $i$ acquires only a block of $\bu_i^{t}$, say the $\ell_i^t$-th block, 
and updates $\hbu_i^t$ as
\begin{align*}
  \hbu_{i,\ell}^{t} = 
  \begin{cases}
    \bu_{i,\ell}^{t}, & \text{if } \ell = \ell_i^t,
    \\[1ex]
    \hbu_{i,\ell}^{t-1}, & \text{if } \ell \neq \ell_i^t,
  \end{cases}
\end{align*}
where, as in~\eqref{eq:perturbation_for_tracking}, $\hbu_{i,\ell}^t$ denotes the $\ell$-th 
block of $\hbu_{i}^t$.
That is, vector $\hbu_i^t$ collects agent $i$'s most recent information on $\bu_i^t$. 
The modified block-tracking scheme then reads  
\begin{align*}
\begin{split}
	 \phi_{(i,\ell)}^{t+1}  & = \sum_{j\in \innbrs_{i,\ell}^t }
   a_{ij\ell}^t \,\phi_{(j,\ell)}^t,
   \\
   \bz_{(i,\ell)}^{t+1} 
   & = \sum_{j\in \innbrs_{i,\ell}^t } 
   \dfrac{a_{ij\ell}^t}{\phi_{(i,\ell)}^{t+1}} 
   \, 
   \Big( 
   \phi_{(j,\ell)}^{t} \bz_{(j,\ell)}^t
   +
   (\hbu_{j,\ell}^{t+1} - \hbu_{j,\ell}^t)
   \Big).
\end{split}
\end{align*}

\section{\algname/ Distributed Algorithm}
\label{sec:algorithm}

In this section we introduce our distributed big-data optimization algorithm
(cf.~Section~\ref{sec:algorithm_description}) along with its convergence
properties (cf.~Section~\ref{sec:algorithm_convergence}). Some extensions of the
basic scheme are discussed in Section~\ref{sec:extensions}.

\subsection{Algorithm Description}
\label{sec:algorithm_description}

The proposed distributed algorithm takes inspiration from two existing optimization
algorithms, namely: NEXT (in-Network succEssive conveX approximaTion)
\cite{dilorenzo2015distributed,dilorenzo2016next} and SONATA (distributed Successive cONvex Approximation
algorithm over Time-varying digrAphs) \cite{sun2016distributed,sun2017distributed}.
These algorithms combine successive convex approximation techniques with a
distributed gradient tracking mechanism to solve convex and nonconvex
optimization problems over time-varying (di)graphs. Specifically, they consist
of a two-step procedure in which each agent: (i) solves a local strongly convex
approximation of the target optimization problem, and (ii) runs a twofold
averaging scheme to reach consensus among the local solution estimates and to
``track'' the average of the gradient of agents' cost functions (the smooth
part).
As all the other existing schemes, they are not designed to solve big-data
optimization problems over networks: they require that, at every iteration,
agents solve a huge-scale optimization problem and communicate their entire
solution estimate to neighbors.

We propose a distributed algorithm, named \algname/, based on a block-wise 
execution of steps (i) and (ii) above. It copes with big-data optimization problems
by unlocking for the first time block-wise optimization and communications.
While the intuitive idea behind this block extension might look simple, we will
show that the convergence analysis of \algname/ is quite challenging. Indeed it
calls for new techniques to deal with local inexact (block-wise) optimization
and communications, the latter inducing block-dependent time-varying digraphs in
the consensus updates.
\algname/ reads as follows.  Each agent maintains a local solution estimate
$\bx_{(i,:)}^t \in\real^{dB}$ of problem \eqref{eq:problem}, with the same block
structure as the optimization variable $\bx$, with
$\bx_{(i,\ell)}^t \in\real^{d}$ being its $\ell$-th block-component. All these
estimates are iteratively updated with the goal of being asymptotically
consensual to a stationary point of problem \eqref{eq:problem}.
Agents also update a local auxiliary variable $\by_{(i,:)}^t \in\real^{dB}$ that
is meant to track $\frac{1}{N}\sum_{j=1}^N \nabla f_j (\bx_{(j,:)}^t) $ (which
is not known locally by the agents), i.e., to get, for any agent $i$,
$\lim_{t\to\infty} \| \by_{(i,:)}^t - \frac{1}{N}\sum_{j=1}^N \nabla f_j
(\bx_{(j,:)}^t) \|=0$.
The update of the $\bx$- and $\by$-variables is described next.

\noindent\textbf{Block-wise local optimization step.}
At  iteration $t$, every agent $i$ selects a block $\ell_i^t \in \until{B}$
according to an essentially cyclic rule satisfying
Assumption~\ref{ass:block_selection}.
As for the optimization step, agent $i$ computes a descent direction
with respect to the selected block (only) by solving a strongly convex approximation of
problem~\eqref{eq:problem} (based on its current solution and gradient
estimates, respectively $\bx_{(i,:)}^t$ and $\by_{(i,:)}^t$).
Specifically, it solves
\begin{align*}
\tildex_{(i,\ell_i^t)}^t & = \argmin_{ \bx_{\ell_i^t} \in \KK_{\ell_i^t} } \:
\surr_{i,\ell_i^t} \big( \bx_{\ell_i^t}; \bx_{(i,:)}^t, \by_{(i,\ell_i^t)}^{t}
\big) + r_{\ell_i^t} ( \bx_{\ell_i^t} ),
\end{align*}
with
\begin{align*}
\begin{split}
  \surr_{i,\ell} \big( \bx_{\ell}; \bx_{(i,:)}^t, \by_{(i,\ell)}^{t} \big) 
  & =
  \tf_i (\bx_\ell; \bx_{(i,:)}^t )
  \\
  & +
  (N \by_{(i,\ell)}^t \! -\! \nabla_\ell f_i ( \bx_{(i)}^t)
  )^{\!\top\!} (\bx_\ell \!-\! \bx_{(i,\ell)}^t ),
\end{split}
\end{align*}
where $\tf_i (\bx_\ell; \bx_{(i,:)}^t )$ is a strongly convex approximation of
$f_i$ satisfying the following assumption.
\begin{assumption}[On the surrogate functions]
Given problem~\eqref{eq:problem} under Assumption~\ref{ass:cost_functions}, 
each surrogate function 
$\tf_{i,\ell}: \KK_\ell\times \KK\rightarrow \real$ is chosen so that %
\begin{enumerate}
\item 
  $\tf_{i,\ell} (\bullet;\bx)$ is uniformly strongly convex with constant $\tau_i >0$
  on $\KK_\ell$;
\item
  $\nabla \tf_{i,\ell} (\bx_{\ell};\bx) = \nabla_\ell f_i (\bx)$, for all $\bx \in \KK$;
\item
  $\nabla \tf_{i,\ell} (\bx_{\ell}; \bullet)$ is uniformly Lipschitz continuous on $\KK$;
\end{enumerate}
where $\nabla \tf_{i,\ell}$ denotes the partial gradient of $\tf_{i,\ell}$ 
with respect to its first argument.~\oprocend
\label{ass:surrogate}
\end{assumption}
 
Several choices for $\tf_i$ are possible; we refer the interested reader to
\cite{facchinei2015parallel,dilorenzo2016next,sun2017distributed,notarnicola2017cdc,notarnicola2017camsap}
for more details and examples.
We point out that each strongly convex function
$\surr_{i,\ell} \big( \bx_{\ell}; \bx_{(i,:)}^t, \by_{(i,\ell)}^{t} \big)$
satisfies $\nabla \surr_{i,\ell} \big( \bx_{(i,\ell)}^t; \bx_{(i,:)}^t, \by_{(i,\ell)}^{t}
\big) = N\by_{i,\ell}^t$,
thus it asymptotically encodes first order information of $\sum_i f_i$, namely
$\sum_{i=1}^N \nabla f_i (\bx_{(i,:)}^t)$.
As a clarifying example, one can
consider the simplest first order approximation of $f_i$
given by its linearization about the 
current iterate $\bx_{(i,:)}^t$,
\begin{align*}
  \surr_{i,\ell} \big( \bx_{\ell}; \bx_{(i,:)}^t, \by_{(i,\ell)}^{t} \big) 
  \!\!=\!
  (\! N \by_{(i,\ell)}^t )^{\!\top \!} (\bx_{\ell} \!-\! \bx_{(i,\ell)}^t) 
  \!+\!
  \tau_i \| \bx_{\ell} \!-\! \bx_{(i,\ell)}^t \|^2\!\!.
\end{align*}

Given $\tildex_{(i,\ell_i^t)}^t $, agent $i$ computes and broadcasts to
its neighbors the feasible point
$\bx_{(i,\ell_i^t)}^t + \gamma^t \Deltax_{(i,\ell_i^t)}^t$, with
$\Deltax_{(i,\ell_i^t)}^t = \tildex_{(i,\ell_i^t)}^t - \bx_{(i,\ell_i^t)}^t$
being a local feasible descent direction and $\gamma^t$ a step-size.

We want to stress that agent $i$ does not optimize, and thus does not
communicate, the other blocks with indexes $\ell \neq \ell_i^t$. For the sake of
analysis, we set $\Deltax_{(i,\ell)}^t = \0$ for the non-updated blocks.

\noindent\textbf{Block-wise averaging and gradient tracking step.}
As for the consensus steps, agent $i$ collects all the updated blocks from its
neighbors and runs two instances of the block-wise perturbed push-sum consensus
scheme, described in Section~\ref{sec:block_consensus}
(Cf. eq.~\eqref{eq:perturbed_block_consensus}). The first one is meant to make
the local solution estimates, $\bx_{(i,:)}^t$, consensual toward their average;
the second, involving a local gradient estimate $\by_{(i,:)}^t$, serves as a
tracking scheme for the gradient signal
$\sum_{i=1}^N \nabla f_i (\bx_{(i,:)}^t)$.

The \algname/ distributed algorithm is summarized (from the perspective of node
$i$) in the next table. 
\begin{algorithm}[!ht]
  \begin{algorithmic}
    \smallskip
    \StatexIndent[0] \textbf{Initialization}: 
      $\bx_{(i,:)}^0 \in \KK$ arbitrary and $\by_{(i,:)}^0 = \nabla f_i(\bx_{(i,:)}^0)$\vskip 0.1cm

    \StatexIndent[0] \textbf{Local Optimization}: \vskip 0.1cm

    \StatexIndent[0.25] select $\ell_i^t \in \until{B}$ and compute
      \begin{align}
        \begin{split}
          \tildex_{(i,\ell_i^t)}^t 
          & =
          \argmin_{ \bx_{\ell_i^t} \in \KK_{\ell_i^t} } \:
          \surr_{i,\ell_i^t} \big( \bx_{\ell_i^t}; \bx_{(i,:)}^t, \by_{(i,\ell_i^t)}^{t} \big)
          +
          r_{\ell_i^t} ( \bx_{\ell_i^t} )
        \end{split}
        \label{eq:alg_local_min}
        \\
        \Deltax_{(i,\ell)}^{t} & = 
        \begin{cases}
        \tildex_{(i,\ell_i^t)}^t - \bx_{(i,\ell_i^t)}^t, & \text{ if } \ell = \ell_i^t,
        \\
        \0, & \text{ otherwise}.
        \end{cases}
        \label{eq:alg_descent_dir}
      \end{align}

    \smallskip
    
    \StatexIndent[0] \textbf{Averaging and Gradient Tracking}: \vskip 0.1cm

    \StatexIndent[0.2]
    For each $j \in \innbrs_i$ receive $\phi_{(j,\ell_j^t)}^{t}$ and 
    $\bx_{(j,\ell_j^t)}^{t} + \gamma^t \Deltax_{(j,\ell_j^t)}^t$.
    \StatexIndent[0.2] For each $\ell \in \until{B}$ compute
    \begin{align}
      \phi_{(i,\ell)}^{t+1} 
      & = \sum_{j\in \innbrs_{i,\ell}^t } a_{ij\ell}^t \, \phi_{(j,\ell)}^{t}
      \label{eq:alg_weights}
      \\
      \bx_{(i,\ell)}^{t+1} & = 
      \sum_{j\in \innbrs_{i,\ell}^t } 
      \frac{a_{ij\ell}^t \,\phi_{(j,\ell)}^{t} }{\phi_{(i,\ell)}^{t+1}}
      \Big(
        \bx_{(j,\ell)}^{t} 
        +
        \gamma^t \Deltax_{(j,\ell)}^t
      \Big)
    \label{eq:alg_x_consenus}
    \end{align}

    \StatexIndent[0.2] For each $j \in \innbrs_i$ receive
    $\big(\phi_{(j,\ell_j^t)}^{t} \by_{(j,\ell_j^t)}^t + \nabla_{\ell_j^t} f_j
    \big( \bx_{(j,:)}^{t+1} \big) - \nabla_{\ell_j^t} f_j \big(
    \bx_{(j,:)}^{t} \big) \big)$
    \StatexIndent[0.2] For each $\ell \in \until{B}$ compute
    \begin{align}
      \!\!\! \by_{(i,\ell)}^{t+1}  \!
      & = \!\!
      \sum_{j \in \innbrs_{i,\ell}^t }\!\!\!
      \frac{a_{ij\ell}^t}{\phi_{(i,\ell)}^{t+1}}
      \Big( 
        \phi_{(j,\ell)}^{t}
        \by_{(j,\ell)}^t 
        \!+\!
        \nabla_\ell f_j \big( \bx_{(j,:)}^{t+1} \big)
        \!-\!
        \nabla_\ell f_j \big( \bx_{(j,:)}^{t} \big) \!
      \Big).
    \label{eq:alg_y_consensus}
    \end{align}
        
  \end{algorithmic}
  \caption{\algname/}
  \label{alg:algorithm}
\end{algorithm}

\begin{remark}
  We would like to stress that agents send \emph{only one} block per iteration. 
  That is, the for-loop over $\ell$ consists of at most $|\nbrs_i|$ non-trivial 
  consensus steps. Thus, each agent $i$ receives 
  exactly $|\nbrs_i| (2d+1)$ updated quantities. Moreover,
  due to the presence of the weights $a_{ij\ell}^t$, each non-trivial consensus step
  requires to sum at most $| \nbrs_i |$ terms over all the blocks.~\oprocend
\end{remark}

\subsection{Algorithm Convergence}
\label{sec:algorithm_convergence}

We now provide the main convergence result of \algname/. 
We first introduce the following assumption on the step-size sequence
$\{\gamma^t\}_{t\ge0}$ [cf. \eqref{eq:alg_x_consenus}].

\begin{assumption}[On the step-size]
  The sequence $\{\gamma^t\}_{t \ge 0}$, with each $0< {\gamma^t} \le 1$, satisfies:
  \begin{itemize}
    \item[(i)]
$\gamma^{t+1} \le \gamma^t$, for all $t\geq 0$;
   \smallskip 
    \item[(ii)]
    $\sum\limits_{t=0}^{\infty}\gamma^t = \infty$ and 
    $\sum\limits_{t=0}^{\infty} (\gamma^t)^2 < \infty$.
  \oprocend
  \end{itemize}
\label{ass:step-size}
\end{assumption}

The above conditions are standard and satisfied by most practical diminishing step-size rules. 
For example, the following rule, proposed in~\cite{facchinei2015parallel}, satisfies 
Assumption~\ref{ass:step-size} and has been found very effective in our 
experiments: $\gamma^{t+1} = \gamma^t ( 1 - \mu \gamma^t )$, 
with $\gamma^0\in ( 0,1 ]$ and  $\mu\in (0,1/\gamma^0)$.

We are now in the position to state the main convergence result, as given below.

\begin{theorem}
  Let  $\{(\bx_{(i,:)}^t)_{i=1}^N\}_{t\ge0}$ %
  be the sequences generated by \algname/ and consider their weighted average 
  \begin{align*}
    \avgdec^t = \frac{1}{N} \Big ( \sum_{i=1}^N \phi_{(i,\ell)}^t \bx_{(i,\ell)}^t \Big )_{\ell=1}^B.
  \end{align*}
  
  Suppose that Assumptions~\ref{ass:cost_functions}, \ref{ass:strong_conn},
  \ref{ass:col_stoch}, \ref{ass:block_selection}, \ref{ass:surrogate} and
  \ref{ass:step-size} are satisfied; 
  then the following statements hold true:
  \begin{itemize}[leftmargin=0.65cm]
    \item[(i)] 
    \texttt{consensus}: $\| \bx_{(i,:)}^t - \avgdec^t \| \to 0$ as $t\to \infty$, for all $i\in\until{N}$;

    \item[(ii)]
    \texttt{convergence}: $\{\avgdec^t\}_{t\ge0}$ is bounded and every of its limit points  
    is a stationary solution of problem~\eqref{eq:problem}. %
\end{itemize}
  \label{thm:convergence}
\end{theorem}
\begin{proof}
See the Appendix.
\end{proof}

Theorem~\ref{thm:convergence} states two results. 
First, a consensus is asymptotically achieved among the local 
estimates $\bx_{(i,:)}^t$ over all the blocks.
Second, the weighted average estimate $\avgdec^t$ converges to the set $\cal S$ of stationary solutions of 
problem~\eqref{eq:problem}. 
Therefore, the sequence $\{(\bx_{(i,:)}^t)_{i=1}^N\}_{t\ge 0}$ converges 
to the set $\{\1_N \kron \bx^\ast \, : \bx^\ast \in \cal S\}$. 

\begin{remark}[Convex Problems]
  If $U$ in~\eqref{eq:problem} is convex, \algname/ converges (in the
  aforementioned sense) to the set of global optimal solutions of the convex
  problem.\oprocend
\end{remark}

\subsection{Alternative Formulations and Generalizations}
\label{sec:extensions}
In this subsection, we discuss some extensions and generalizations of the basic \algname/.
First, we start by describing a special instance for an unconstrained version of
problem~\eqref{eq:problem} with all $\reg_\ell = 0$.
If one chooses the simplest surrogate in~\eqref{eq:alg_x_update}, namely the 
linearization of $f_i$ about the current iterate, then \algname/ reads
\begin{align*}
  \phi_{(i,\ell)}^{t+1} 
  & = 
  \sum_{j\in \innbrs_{i,\ell}^t } a_{ij\ell}^t \, \phi_{(j,\ell)}^{t}
  \\
  \bx_{(i,\ell)}^{t+1} 
  & = 
  \sum_{j\in \innbrs_{i,\ell}^t } 
      \frac{a_{ij\ell}^t \,\phi_{(j,\ell)}^{t} }{\phi_{(i,\ell)}^{t+1}}
      \Big(
        \bx_{(j,\ell)}^{t} 
        -
        \gamma^t \by_{(j,\ell)}^t
      \Big)
  \\
  \by_{(i,\ell)}^{t+1} 
  & = 
  \sum_{j \in \innbrs_{i,\ell}^t} \frac{a_{ij\ell}^t}{\phi_{(i,\ell)}^{t+1}} 
  \Big( \!\phi_{(j,\ell)}^{t} \, \by_{(j)}^{t}
  \!\!+\! \nabla_\ell f_j (\bx_{(j,:)}^{t+1}) \!-\! \nabla_\ell f_j (\bx_{(j,:)}^{t}) \!\Big),
\end{align*}
which is a block-wise implementation of existing distributed algorithms
based on a gradient tracking scheme as, 
e.g.,~\cite{dilorenzo2016next,nedic2017achieving,qu2017harnessing,
sun2017distributed,xi2018addopt,xu2018convergence,xin2018linear}.

\noindent\textbf{Combine-Then-Adapt Averaging.}
The block-wise consensus and tracking updates as in~\eqref{eq:alg_x_consenus}
and~\eqref{eq:alg_y_consensus} are performed in the so-called Adapt-Then-Combine
(ATC) fashion. We remark that they can be also performed adopting the other scheme 
used in the literature, namely the so-called Combine-Then-Adapt (CTA) way 
\cite{sayed2014adaptation}.
The CTA form of the averaging and gradient tracking step 
of \algname/ reads
\begin{align*}
  \bx_{(i,\ell)}^{t+1} 
  & = 
  \sum_{j\in \innbrs_{i,\ell}^t } \!\!
  \frac{a_{ij\ell}^t\phi_{(j,\ell)}^{t}}{\phi_{(j,\ell)}^{t+1}}\, \bx_{(j,\ell)}^{t} 
  +
  \gamma^t \phi_{(i,\ell)}^{t} \Deltax_{(i,\ell)}^t
  \\
  \by_{(i,\ell)}^{t+1}
  & = 
  \sum_{j \in \innbrs_{i,\ell}^t } \!\!\!
  \frac{a_{ij\ell}^t \, \phi_{(j,\ell)}^t }{\phi_{(i,\ell)}^{t+1}} 
  \by_{(j,\ell)}^t 
  +\frac{
  \nabla_\ell f_i \big( \bx_{(i,:)}^{t+1} \big)
  \!-\!
  \nabla_\ell f_i \big( \bx_{(i,:)}^{t} \big) 
  }{\phi_{(i,\ell)}^{t+1}}.
\end{align*}
One can show that Theorem~\ref{thm:convergence} also applies to the CTA form of \algname/, which 
thus converges under the same condition of its ATC counterpart.

\noindent\textbf{Block-Wise Gradient Computation.}
In order to perform~\eqref{eq:alg_y_consensus}
(and also its CTA counterpart), agent $i$ needs to compute the entire gradient
$\nabla f_i(\bx_{(i,:)}^{t+1})$ [recall from \eqref{eq:weights} that
$a_{ii\ell}>0$, for all $\ell$]. This potential drawback can be overcome  considering 
a slightly different version of \algname/ in which $\nabla_\ell f_i$ is
replaced by an auxiliary variable $\hbg_{(i,\ell)}$, which is iteratively
updated as 
\begin{align*}
  \hbg_{(i,\ell)}^{t+1}
  & =
  \begin{cases}
    \nabla_{\ell_i^t} f_i \big( \bx_{(i,:)}^{t+1} \big), \hspace{0.3cm} & \text{if } \ell = \ell_i^t,
    \\
    \hbg_{(i,\ell)}^{t}, & \text{otherwise.}
  \end{cases}
\end{align*}
Thus, step~\eqref{eq:alg_y_consensus} must be replaced by
\begin{align*}
  \by_{(i,\ell)}^{t+1}
  & = 
  \sum_{j \in \innbrs_{i,\ell}^t } \!\!
  \frac{a_{ij\ell}^t \phi_{(j,\ell)}^t}{\phi_{(i,\ell)}^{t+1}} \by_{(j,\ell)}^t 
  +
  \frac{
  \hbg_{(i,\ell)}^{t+1}
  -
  \hbg_{(i,\ell)}^{t}}{\phi_{(i,\ell)}^{t+1}}.
\end{align*}

\begin{remark}
  The auxiliary mechanism of $\hbg$ imposes that, at each iteration $t$ of
  \algname/, each agent $i$ computes two components of the same gradient,
  $\nabla f_i (\bx_{(i,:)}^t)$, rather than one. This twofold computation can be
  avoided by using a slight modification of the scheme in which the block index
  is selected after the optimization step, see~\cite{notarnicola2017camsap} for
  further details.  \oprocend
\end{remark}

\noindent\textbf{Time-Varying Communication Digraph.}
In Section~\ref{sec:block_consensus}, we have assumed that agents communicate
according to a fixed, strongly connected digraph $\GG$. 
However, even if the starting communication network is static, the block
selection rule gives rise to time-varying digraphs $\GG_\ell^t$.
In fact, for the proposed algorithm to work we just need the induced digraph
sequences $\{ \GG_\ell^t\}_{t \ge 0}$ to be $T$-strongly connected. 
Thus, \algname/ immediately applies to a set-up in which agents communicate
according to a time-varying communication digraph $\{\GG^t\}_{t \ge 0}$ (with
associated column stochastic matrix $\tilde{\bA}^t$), provided that the
essentially cyclic rule applied to the time-varying digraph satisfies the
following assumption.

\begin{assumption}
\label{ass:tv_strong_conn}
There exist $T>0$, such that each digraph sequence $\{ \GG_\ell^t\}_{t \ge 0}$
is $T$-strongly connected, for all $\ell\in\until{B}$.\oprocend
\end{assumption}

Specifically, Theorem~\ref{thm:convergence} holds if
Assumption~\ref{ass:strong_conn} is replaced by
Assumption~\ref{ass:tv_strong_conn}.

\section{Numerical Study:\\ Application to Sparse Regression}
\label{sec:simulations}
In this section we apply \algname/ to the distributed sparse regression
problem. Consider a network of $N$ agents taking linear measurements of a sparse
signal $\bx_0\in \real^m$, with measurement matrix
$\bD_i\in \real^{n_i \times m}$.  The observation taken by agent $i$ can be
expressed as $\bb_i = \bD_i\bx_0 + \bn_i$, where $\bn_i\in \real^{n_i}$ accounts
for the measurement noise.  The estimation of the underlying signal $\bx_0$ is
obtained solving the following problem
\begin{align}
  \min_{\bx \in \KK} \: & \: 
  \sum_{i=1}^N \: \underbrace{\| \bD_i \bx - \bb_i \|^2_2}_{ f_i(\bx) }
  \: + \: 
  \reg(\bx),
\label{eq:regression}
\end{align}
where $\bx \in \real^{m}$; $\KK$ is the box constraint set $\KK \triangleq [k_L,k_U]^m$, 
with $k_L\leq k_U$; 
and $\map{\reg}{\real^m}{\real}$ is a difference-of-convex (DC) sparsity-promoting regularizer, 
given by 
\begin{equation*}
  \reg(\bx) 
  \triangleq 
  \lambda \cdot 
  \sum_{j=1}^m r_0(x_j),\quad r_0(x_j) 
  \triangleq \frac{\log(1+\theta |x_j|)}{\log(1+\theta)},
\end{equation*}
where $\lambda$ and $\theta$ are positive tuning parameters.

The first step to apply \algname/ is to build a valid surrogate $\tilde{f}_{i,\ell}$ of $f_i$ (cf. Assumption \ref{ass:surrogate}). To this end, we first rewrite $r_0$ as a difference-of-convex function. It is not difficult 
to check that 
\begin{equation*}
  r_0(x) 
  = 
  \underbrace{\eta(\theta) \, |x|}_{r_0^+(x)} 
  - 
  \underbrace{ \big ( \eta(\theta)\,|x| - r_0(x) \big) }_{r_0^-(x)},
\end{equation*}
where $r_0^+ : \real \to \real $ is convex non-smooth with 
$\eta(\theta ) \triangleq \theta/\log(1+\theta)$,
and $\map{r_0^-}{\real}{\real}$ is convex with Lipschitz continuous first 
order derivative given by  
\begin{equation*}
  \frac{d r_0^-}{dx} (x)
  = 
  \sign(x)\cdot \frac{\theta^2 |x|}{\log(1+\theta)(1+\theta |x|)}.
\end{equation*}

Denoting the coordinates associated with the $\ell$-th block as $\II_\ell \subset \until{B}$, let us 
define the matrix $\bD_{i,\ell}$ [resp. $\bD_{i,-\ell}$] constructed by picking the 
columns of $\bD_i$ that belongs [resp. does not belong] to $\II_\ell$. 
Then, the following is a valid surrogate function for each agent $i$
that satisfy Assumption \ref{ass:surrogate}. We
consider $\tf_i$ obtained as the linearization of $f_i$ and $-r_0^-$, about the current 
solution estimate, which leads to 
\begin{equation*}
\begin{split}
  &
  \tf_{i,\ell}(\bx_{\ell};\bx_{(i,:)}^t) 
  = 
  \left(2\bD_{i,\ell}^\top (\bD_i - \bb_i)\right)^{\!\!\top}\! (\bx_{\ell} - \bx_{(i,\ell)}^t) 
  \\
  & 
  \hspace{0.8cm}
  +
  \frac{\tau_i }{2} \| \bx_{\ell} - \bx_{(i,\ell)}^t\|^2
  - 
  \!\! \sum_{k \in \II_\ell} 
  \underbrace{ \frac{d r_0^-((\bx_{(i,\ell)}^t)_k ) }{dx} }_{v_{ik}^t }
  ( \bx_\ell - \bx_{(i,\ell)}^t)_k,
\end{split}
\end{equation*}
where $x$ is a scalar variable and, e.g., $(\bx^t_{(i,\ell)})_k$ denotes the 
$k$-th scalar component of $\bx^t_{(i,\ell)}$.
Note that the minimizer of $\tilde{f}_{i,\ell}$ can be computed in closed form, and is 
given by
\begin{equation*}
  \bx_{(i,\ell)}^{t+1} = \mathcal{P}_{\KK_\ell} \Big( \SS_{\frac{\lambda \eta}{\tau_i } } 
  \Big ( \bx_{(i,\ell)}^t - \frac{1}{\tau_i}(2\, \bD_{i,\ell}^\top (\bD_i - \bb_i) - \bv^t_{i,\ell})
  \Big) 
  \Big)
\end{equation*}
where $\bv^t_{i,\ell} \triangleq (v_{ik}^t)_{k\in \II_\ell}$, 
$\SS_\lambda(\bx) \triangleq \sign (\bx) \cdot \max\{ |\bx| - \lambda,0\}$ 
(operations are performed element-wise), and $\mathcal{P}_{\KK_\ell}$ is the Euclidean projection  onto $\KK_\ell$.

We test our algorithm, considering the following  simulation set-up.  
The variable dimension $m$ is set to be $400$, $\KK$ is set to be $[-10,10]^{400}$, and the regularization 
parameters are set to $\lambda=0.15$
and $\theta=7$. The network is composed of $N=30$ agents, communicating over an undirected 
graph $\GG$, obtained using an Erd\H{o}s-R\'enyi random model. We considered
two extreme network topologies: a densely and a poorly connected one, which have algebraic 
connectivity equal to $25$ and $5$, respectively.
The components of the ground-truth signal $\bx_0$ are i.i.d.,  generated according to the Normal 
distribution $\NN(0,1)$. To impose sparsity on $\bx_0$, we set the smallest $80\%$ of the 
entries of $\bx_0$ to zero.
Each agent $i$ has a measurement matrix $\bD_i \in \real^{300\times 400}$ with 
i.i.d.~$\NN(0,1)$ distributed entries (with $\ell_2$-normalized rows), and 
the observation noise $\bn_i \in \real^{300}$ has entries i.i.d. distributed according 
to $\NN (0,0.5)$.
 
We compare our algorithm with the (sub)gradient-projection algorithm 
proposed in~\cite{bianchi2013convergence}. 
Note that there is no formal proof of convergence for such an algorithm in the 
nonconvex setting; moreover it is designed for the non-block-wise case, i.e., $B=1$. 
We used the following tuning for the algorithms. 
The diminishing step-size is chosen as $\gamma^t = \gamma^{t-1} (1 - \mu \gamma^{t-1})$,
with $\gamma^0 = 0.3$ and $\mu =10^{-3}$; the proximal parameter $\tau_i = 10$
for all $i$.
To evaluate the algorithmic performance we used three merit functions.
The first one measures the distance from stationarity
of the average of the agents' iterates 
$ \avgdec^t = \frac{1}{N} ( \sum_{i=1}^N \phi_{(i,\ell)}^t \bx_{(i,\ell)}^t )_{\ell=1}^B$, 
and is given by 
\begin{align*}
	J^t \triangleq 
	\bigg \|
	\avgdec^t - 
	\mathcal{P}_{\KK}
	  \Big(\SS_{\lambda\eta} 
	  \Big( 
	    \avgdec^t  - \Big(\smallsum_{i=1}^N \nabla f_i (\avgdec^t ) - \reg (\avgdec^t) \Big)
	    \Big) \Big)
	\bigg \|_\infty.
\end{align*}
Note that $J^t$ is a valid merit function: it is continuous and it is zero if and only if the $\avgdec^t$ is a stationary 
solution of problem~\eqref{eq:regression}. 
The other two merit functions quantify the consensus disagreement at each iteration among the
solution estimates and the trackers. They are defined as
 \begin{align*}
  D^t & \triangleq \max_{i\in\until{N}} \| \bx_{(i,:)}^t - \avgdec^t \|,
  \\
  R^t & \triangleq \max_{i\in\until{N}} \| \by_{(i,:)}^t - \avggrad^t \|,
\end{align*}
where the average $\avgdec^t$ is defined as before, while the average tracker
is $\avggrad^t = \frac{1}{N} ( \sum_{i=1}^N \phi_{(i,\ell)}^t \by_{(i,\ell)}^t )_{\ell=1}^B$.

The performance of \algname/ for different choices of the block dimension $B$ 
are reported in Figure~\ref{fig:convergence_rate}.
To fairly compare the algorithms run for different block sizes, we plot $J^t$, $D^t$ and $R^t$
versus the average agents' ``message exchanges'', defined as $t/B$, where $t$ is the
iteration counter used in the algorithm description.
The figures show that stationarity, consensus and correct tracking have been achieved by \algname/
within $200$ message exchanges while the plain gradient scheme \cite{bianchi2013convergence} 
is much slower. 
\begin{figure}[!htbp]
  \centering
	\includegraphics[scale=1]{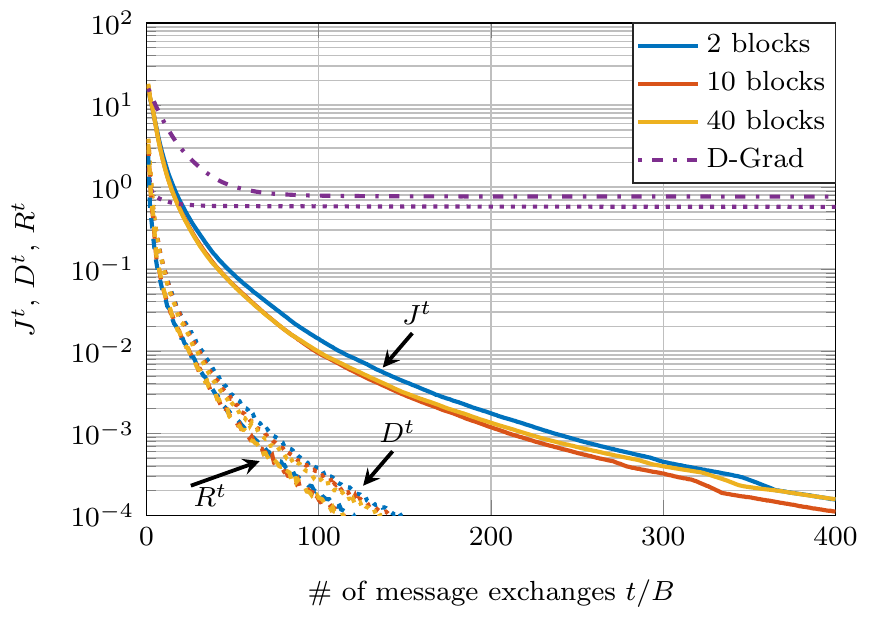}
\caption{
  Optimality measurement $J^t$ (solid), consensus error $D^t$ (dotted)
  and tracking error $R^t$ (dashed)
  versus the number of message exchange for several choices of the number of blocks $B$.
  }
\label{fig:convergence_rate}
\end{figure}

\noindent
Let $t_{\text{end}}$ be the completion time up to a tolerance $10^{-3}$, i.e., the 
iteration counter of the distributed algorithm such that $J^{t_{\text{end}}} \!<\! 10^{-3}$. 
Fig.~\ref{fig:blk_tradeoff} shows the normalized completion time $t_{\text{end}}/B$ 
versus the number of blocks $B$. It highlights how the communication cost 
reduces by increasing the number of blocks.

\begin{figure}[!htbp]
  \centering
	\includegraphics[scale=1]{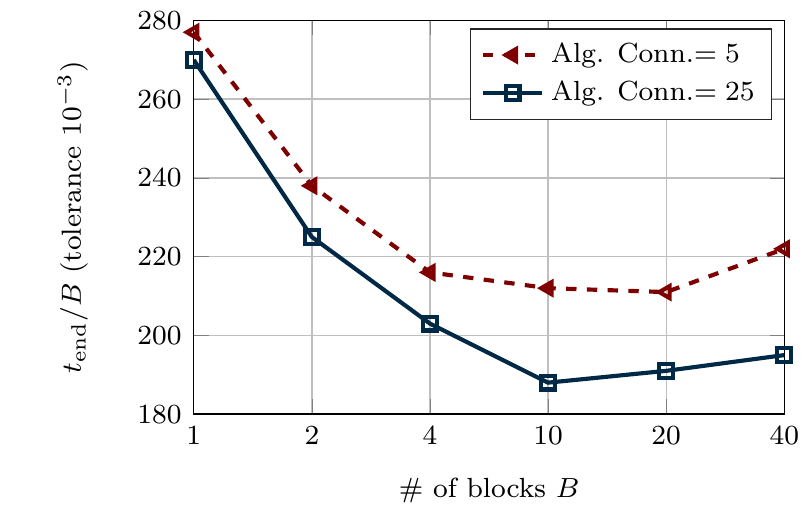}
\caption{
  Completion time required to obtain $J^t < 10^{-3}$ versus the number of blocks $B$
  for two network topologies.
  }
\label{fig:blk_tradeoff}
\end{figure}

\section{Conclusions}
\label{sec:conclusions}
  In this paper we proposed a novel block-iterative distributed scheme for
  nonconvex, big-data optimization problems over networks. That is, we addressed
  large-scale optimization problems in which the dimension of the decision
  vector is huge via a distributed algorithm (over network) in which each agent
  optimizes over and communicates one block only of the entire decision vector.
  Specifically, at each iteration, agents solve a local optimization problem
  (involving only one block of the decision vector) in which a strongly convex
  approximation of the global (possibly nonconvex) cost function is
  minimized. The optimization step is combined with a novel block-wise perturbed
  consensus protocol based on the communication to neighboring agents of one
  block only. This scheme is applied to the local solution estimates and to a
  local vector estimating the gradient of the (smooth part of the) global cost
  function. We proved that agents achieve consensus to their (weighted) average,
  and that any limit point of the average sequence is a stationary solution of
  the optimization problem. Finally, we provided numerical results
  corroborating our theoretical findings and highlighting the impact of the block
  dimension on algorithm performance.  

\appendix
\renewcommand{\thetheorem}{\thesubsection.\arabic{theorem}}

\section{Convergence Analysis}
\label{sec:analysis}
To study convergence of \algname/, it is convenient to introduce some auxiliary variables, namely: 
$\bs^t_{(i,:)}\triangleq (\bs^t_{(i,\ell)})_{\ell=1}^B$ and 
$\bsigma^t_{(i,:)}\triangleq (\bsigma^t_{(i,\ell)})_{\ell=1}^B$,
for all $i\in\until{N}$.
Steps~\eqref{eq:alg_weights}, \eqref{eq:alg_x_consenus}, and \eqref{eq:alg_y_consensus} in \algname/ can 
be then rewritten as: for all $\ell \in \until{B}$ and $i \in \until{N}$,
\begin{align}
  \phi_{(i,\ell)}^{t+1} & = \smallsum_{j \in \innbrs_i } a_{ij\ell}^t \, \phi_{(j,\ell)}^t,
  \label{eq:alg_phi_update}
  \\
  \bs_{(i,\ell)}^{t+1} & = 
  \smallsum_{j\in \innbrs_i } a_{ij\ell}^t 
      \Big(
        \bs_{(j,\ell)}^{t} 
        +
        \gamma^t 
        \phi_{(j,\ell)}^t
        \Deltax_{(j,\ell)}^t%
      \Big ),
  \label{eq:alg_s_update}
  \\
      \bx_{(i,\ell)}^{t+1} & = 
      \dfrac{\bs_{(i,\ell)}^{t+1}}{\phi_{(i,\ell)}^{t+1}},
  \label{eq:alg_x_update}
    \\
      \bsigma_{(i,\ell)}^{t+1} 
      & = \!
      \smallsum_{j \in \innbrs_i}
      \!\!
      a_{ij\ell}^t
      \Big( \!
        \bsigma_{(j,\ell)}^t 
        \!+\!
        \nabla_\ell f_j \big( \bx_{(j,:)}^{t+1} \big)
        \!-\!
        \nabla_\ell f_j \big( \bx_{(j,:)}^{t} \big) \!
      \Big),
  \label{eq:alg_sigma_update}
  \\
  \by_{(i,\ell)}^{t+1} & = 
  \dfrac{\bsigma_{(i,\ell)}^{t+1} }{\phi_{(i,\ell)}^{t+1}},
  \label{eq:alg_y_update}
\end{align}
with each $\bsigma_{(i,:)}^0 \triangleq  \nabla f_i(\bx_{(i,:)}^0)$.

Averaging \eqref{eq:alg_s_update} and~\eqref{eq:alg_sigma_update} 
over $i \in \until{N}$ and using the column stochasticity of each $\bA_{\ell}^t$, 
yields the following dynamics for the block-averages:  for each  $\ell \in \until{B}$,
\begin{align}
  \label{eq:avg_dec_block_evolution}
  \avgdec_\ell^{t+1} 
  & = \avgdec_\ell^t + \gamma^t
  \frac{1}{N} \smallsum_{i=1}^N \phi_{(i,\ell)}^t \Deltax_{(i,\ell)}^t,  
  \\
  \label{eq:avg_grad_block_evolution}
  \avggrad_\ell ^{t+1} 
  & = \avggrad_\ell ^t +
  \frac{1}{N}\smallsum_{i=1}^N \Big( \nabla_{\ell} f_i ( \bx_{(i,:)}^{t+1}) - \nabla_{\ell} f_i ( \bx_{(i,:)}^t ) \Big),
\end{align}
where $\avgdec_\ell^t \triangleq ({1}/{N})\cdot \sum_{i=1}^N \bs_{(i,\ell)}^{t} $ 
and $\avggrad_\ell ^{t} \triangleq ({1}/{N})\cdot\sum_{i=1}^N  \bsigma_{(i,\ell)}^t $. 
We also define $\avgdec^t\triangleq (\avgdec_\ell^t)_{\ell=1}^B$ and 
$\avggrad^t\triangleq (\avggrad_\ell^t)_{\ell=1}^B$. 
To prove Theorem~\ref{thm:convergence},  it is sufficient to show that: 
(i) all the local copies $\bx_{(i,:)}^t$ converge to $\avgdec^t$; 
and (ii) every limit point of $\{ \avgdec^t \}_{t\geq 0}$ is a stationary solution of 
problem~\eqref{eq:problem}. 

Notice that given a linear dynamical system in the 
form~\eqref{eq:avg_dec_block_evolution}, one can always write
\begin{align*}
  \avgdec_\ell^{t+\theta_t} 
  = 
  \avgdec_\ell^{t} + \smallsum_{\tau=t}^{t+\theta_t-1} \bu_\ell^\tau
\end{align*}
for every integer $\theta_t \in [0,T]$, where we used the short-hand $\bu_\ell^\tau  
= \gamma^\tau \frac{1}{N} \sum_{i=1}^N \phi_{(i,\ell)}^\tau \Deltax_{(i,\ell)}^\tau$.
Thus, if the input $\bu_\ell^\tau $ is vanishing, i.e., $ \lim_{\tau\to\infty} 
\| \bu_\ell^\tau \| = 0 $, there holds
\begin{align}
\label{eq:vanishing_avgdec_diff}
\begin{split}
  \lim_{t\to\infty} 
  \| \avgdec_\ell^{t+\theta_t}  - \avgdec_\ell^t \| 
  & 
  =
  \lim_{t\to\infty} 
  \smallsum_{\tau=t}^{t+\theta_t-1} \|\bu_\ell^\tau\|
  \\
  &
  \le
  \lim_{t\to\infty} 
  \smallsum_{\tau=t}^{t+T-1} \|\bu_\ell^\tau\|
  = 0.
\end{split}
\end{align}

\noindent \textbf{Structure of the proof:} 
The proof is organized as follows. In Section~\ref{sec:prelim}, we introduce some preliminary 
results that will be used in the rest of the sections, namely: (i) a formal description of the 
perturbed push-sum algorithm along with its convergence properties; and (ii) a list of key 
properties of a best-response map $\tildex^t$ and related quantities.
Theorem~\ref{thm:convergence}(i) is proven in Section~\ref{sec:consensus_proof}, where 
convergence of the consensus updates~\eqref{eq:alg_x_update} 
and tracking mechanism~\eqref{eq:alg_y_update} is studied. 
More specifically, first we prove that $\lim_{t\to\infty} \| \bx_{ (i,:) }^t - \avgdec^t \| = 0$, 
for all $i \in \until{N}$ (cf.~Proposition~\ref{prop:consensus}), showing thus asymptotic 
consensus of the local estimates $\bx_{(i,:)}^t$; and, second, 
$\lim_{t\to\infty} \| \by_{ (i,:) }^t - \avggrad^t \| = 0$, for all $i \in \until{N}$ 
(cf.~Proposition~\ref{prop:tracker_convergence}), which together with
\begin{equation}
\label{eq:gradient_sum_conservation}
  \avggrad^t = \frac{1}{N} \smallsum_{i=1}^N \nabla f_i(\bx_{(i,:)}^t), \quad \forall t\geq 0, 
\end{equation}
proves that each $\by_{(i,:)}^t$ tracks asymptotically the average of the cost function gradients. 
In Section~\ref{sec:opt},
we study the descent properties of a suitably defined Lyapunov-like function along the 
trajectory $\{(\bx_{(i,:)}^t)_{i=1}^N, \avgdec^t\}_{t\geq 0}$. This result is instrumental to 
show (subsequence) convergence of $\{\avgdec^t\}_{t\geq 0}$ to stationary solutions of 
problem~\eqref{eq:problem} [in the sense of Theorem~\ref{thm:convergence}(ii)], 
which is proven in Section~\ref{sec:convergence_s_bar}.

\subsection{Technical preliminaries} 
\label{sec:prelim}

\subsubsection{Perturbed push-sum consensus} 
Consider a network of $N$ agents communicating,  at each time slot $t$, over the graph 
$\GG^t \triangleq (\until{N},\EE^t)$. The vector form of the  perturbed push-sum protocol 
introduced in~\cite{nedic2015distributed} reads: for all $t\geq 0$,
\begin{align}
\begin{split}
  \psi_i^{t+1} & = \smallsum_{j\in\nbrs_i^t} a_{ij}^t \psi_j^t
  \\
  \boldeta_i^{t+1} & = \smallsum_{j\in\nbrs_i^t} a_{ij}^t ( \boldeta_j^t + \bepsilon_j^t)
  \\
  \bz_i^{t+1} & = \frac{\boldeta_i^{t+1}}{\psi_i^{t+1}},
\end{split}
\label{eq:perturbed_push-sum}
\end{align}
where $\psi_i\in \real$, $\boldeta_i\in \real^n$, $\bz_i\in \real^n$ are agent $i$'s local variables,
with $\psi_i^0 = 1$, and  $\{\bepsilon_i^t\}_{t\geq 0}$ is a given perturbation sequence 
(known by agent $i$ only).
The graph $\GG^t$ and weight matrix $\bA^t \triangleq (a_{ij})_{i,j = 1}^N$ satisfy the 
following assumptions.
\begin{assumption}
\label{ass:T_connectivity}
  The graph sequence $\{\GG^t\}_{t \geq 0}$ is strongly connected, i.e., there 
  exists an integer $T>0$ such that the union digraph $\bigcup_{\tau=0}^{T-1} \GG^{t+\tau} 
  \triangleq (\until{N},  \cup_{\tau=0}^{T-1} \EE^{t+\tau})$ is strongly 
  connected for all $t \geq 0$.\oprocend
\end{assumption}

\begin{assumption}
\label{ass:col_stoch_A}
Each weight matrix $\bA^t$ matches graph $\GG^t$, that is, it satisfies
\begin{itemize}
  \item[(1)] $a_{ij}^t = 0$, if $(j,i) \notin \EE^t$; and $a_{ij}^t \geq \kappa > 0$, if  $(j,i) \in \EE^t$;
  \item[(2)] $a_{ii}^t \geq \kappa > 0$, for all $i \in \until{N}$;
  \item[(3)] $\bA^t$ is column stochastic, i.e., $\1^\top \bA^t = \1^\top$.\oprocend
\end{itemize}
\end{assumption}

The convergence properties of the (scalar version of the) perturbed push-sum protocol 
have been studied in  \cite[Lemma~1]{nedic2015distributed}, as summarized below [for the 
vector case (\ref{eq:perturbed_push-sum})]. 
\begin{lemma} \label{lem:perturbed_pushsum}
  Consider the perturbed push-sum protocol~\eqref{eq:perturbed_push-sum} under 
  Assumptions~\ref{ass:T_connectivity} and \ref{ass:col_stoch_A}.
  Then the following hold:
  \begin{itemize}
    \item[(1)] For all $t\geq 0$, 
\begin{align}
  \label{eq:push_sum_upper_bound}
 \hspace{-0.6cm} \Big \| 
    \bz_i^{t+1} \!-\! \frac{1}{N} \smallsum_{j=1}^N ( \boldeta_j^t + \bepsilon_j^t)
  \Big \| 
  \!\le \!
  c_1 (\rho)^t 
  +
  c_2 \smallsum_{\tau =1}^{t}
  (\rho)^{t-\tau} 
  \sum\limits_{i=1}^N 
  \| \bepsilon_i^\tau \|_1,
\end{align}
where $\rho \in (0,1)$ and $c_1$ and $c_2$ are some positive, finite scalars;

\item[(2)] If the perturbations are vanishing, i.e., $\lim\limits_{t\to\infty} \|\bepsilon_i^t\| = 0$, for all $i\in \until{N}$, then
\begin{align*}
  \lim_{t\to\infty} 
  \Big \| \bz_i^{t+1} - \frac{1}{N} \smallsum_{j=1}^N ( \boldeta_j^t + \bepsilon_j^t) \Big \| 
  = 0;
\end{align*}

\item[(3)] The sequence $\{ \psi_i^t \}_{t \ge 0}$ satisfies 
\begin{equation*}
  \inf_{t \ge 0 } \left( \min_{i \in\until{N}} \psi_i^t\right) \triangleq \delta >0.
  \eqoprocend
\end{equation*}
\end{itemize}
\end{lemma}

Note that, since $\bA^t$ is column stochastic, we have 
$\bar{\boldeta}^{t+1} \triangleq \frac{1}{N} \sum_{j=1}^N \boldeta_j^{t+1} 
= \frac{1}{N} \sum_{j=1}^N ( \boldeta_j^t + \bepsilon_j^t)$. Therefore, the 
bound~\eqref{eq:push_sum_upper_bound} can be written also as
\begin{align}
  \Big \| 
    \bz_i^{t+1} - \bar{\boldeta}^{t+1} 
  \Big \| 
  \le 
  c_1 (\rho)^t 
  +
  c_2 \smallsum_{\tau =1}^{t}
  (\rho)^{t-\tau} 
  \sum\limits_{i=1}^N 
  \| \bepsilon_i^\tau \|_1.
\label{eq:push-sum evolution bound}
\end{align}

\subsubsection{Properties of the best-response map $\tildex^t$} 
\label{app:technicalities}

In this subsection, we introduce some intermediate results dealing with key properties of 
a best-response map $\tildex^t$ and related quantities. For notational simplicity, we state 
the results in a more abstract form, omitting  time and agent index dependencies. 

Consider the following optimization problem
\begin{align}
\label{eq:technicalities}
  \tildex \triangleq \argmin_{\bx \in \KK} \: \: h(\bx) + \reg(\bx)
\end{align}
where $\KK$ is a closed convex set and  $h$ (resp. $\reg$) is a $\CC^1$ (resp. convex, possibly nonsmooth) function on (an open set containing) $\KK$. Given some $\mathbf{w}\in \KK$, let us also introduce the function $\hath(\bullet; \bw, \nabla h (\bw)): \KK \to \KK$ (the explicit dependence of $\hath$ from $\bw$ and  $\nabla h (\bw)$ is immaterial for our discussion). We assume that  $\hath(\bullet; \bw, \nabla h (\bw))$ satisfies  the following conditions:
\begin{itemize}
\item[(1)] $\hath(\bullet; \bw, \nabla h (\bw))$ is $\CC^1$ (on an open set containing $\KK$) and $\tau$-strongly convex on $\KK$;
\item[(2)] $\nabla \hath (\bw; \bw, \nabla h (\bw) ) = \nabla h (\bw)$;
\item[(3)] $\nabla_{\bw} \hath (\bx; \bw, \nabla h (\bw) )$ is uniformly Lipschitz continuous for all $\bx \in \KK$.
\end{itemize}
The function $\hath(\bullet; \bw, \nabla h (\bw))$ should be considered as a strongly convex approximation of $h$ having the same gradient of $h$ at $\mathbf{w}$. Given $\hath(\bullet; \bw, \nabla h (\bw))$, we can finally introduce the following  optimization problem 
\begin{align}
  \label{best-response-hat-x}
  \hatx(\bw) = \argmin_{ \bx \in \KK} \: \hath(\bx; \bw, \nabla h (\bw)) + r(\bx),
\end{align}
which can be considered as a convex approximation of (\ref{eq:technicalities}). 

The following results establish some key properties of the best-response maps $ \tildex$ and $\hatx$.
\begin{lemma}
  Consider problem~\eqref{eq:technicalities} under the further assumption that $h$ is $\tau$-strongly convex. Then, for all $\bv \in \KK$, the following hold:
  \begin{enumerate}
    \item 
    $ 
      \| \tildex  - \bv \|
      \le
      \dfrac{1}{\tau}
      \| \nabla h(\bv)  \|
      +
     \dfrac{1}{\tau}
     \|  \subgrad \reg ( \tildex ) \|
  $;\vskip 0.2cm

\item $  \nabla h(\bv) ^\top ( \tildex - \bv) \le -\tau   \| \tildex  - \bv \|^2
  -(\reg ( \tildex ) -  \reg ( \bv ) ).
$
\end{enumerate}
\label{lem:strongly_convex_minimization}
\end{lemma}
\begin{proof} 
The proof follows readily from the first order optimality conditions of \eqref{eq:technicalities} and the convexity of $r$. 
\end{proof}

\begin{proposition}[\!\!\!\!{\cite[Prop.~8]{facchinei2015parallel}}] 
\label{prop:best_response_map_properties}
The best-response map ${\KK}{\ni}\mathbf{w}\longmapsto \hatx (\bw)$ defined 
in~\eqref{best-response-hat-x} satisfies
\begin{itemize}
    \item[(1)] $\hatx ( \bullet )$ is Lipschitz continuous on $\KK$; 

    \item[(2)] The set of the fixed-points of $\hatx (\bullet)$ coincides with the set of stationary 
    solutions of problem~\eqref{eq:technicalities}; therefore $\hatx (\bullet)$ has a fixed 
    point. \oprocend
  \end{itemize}
\end{proposition}

We can now customize the above results to our setting.  
Consider the best-response $\widetilde{\bx}_{(i,\ell)}^t$ in \eqref{eq:alg_descent_dir};
applying Lemma~\ref{lem:strongly_convex_minimization}(ii) we readily obtain the following.
\begin{lemma}
\label{lem:gradient_related_condition}
The best-response $\widetilde{\bx}_{(i,\ell)}^t$ defined in~\eqref{eq:alg_descent_dir} satisfies
\begin{align}
\label{eq:sufficient_descent} 
\begin{split}
  \big( \by_{(i,\ell)}^{t} \big)^\top 
  \Deltax_{ (i,\ell) }^{t} \le - \tau_i \|\Deltax_{ (i,\ell) }^{t}  \|^2 - \big(r_\ell (\widetilde{\bx}_{(i,\ell)}^t) - r_\ell (\bx_{(i,\ell)}^t)\big),
\end{split}
\end{align}
for all $\ell\in \until{B}$.\oprocend
\end{lemma}

Finally, consider the best-response map ${\KK}\,{\ni}\,\mathbf{w} \longmapsto \hatx_{(i,\ell)} (\bw)$, 
defined as
\begin{equation}
\label{eq:xhat_definition}
 \hspace{-0.3cm}  \hatx_{(i,\ell)} \big(\bw\big) 
  \triangleq 
  \argmin_{\bx_\ell \in \KK_\ell } \:
  \surr_{i,\ell} \Big(
    \bx_\ell ;\bw,  \dfrac{1}{N} \smallsum_{i=1}^N \nabla_\ell f_i (\bw)
  \Big) 
  + r_\ell (\bx_\ell ).
\end{equation}
Clearly \eqref{eq:xhat_definition} is an instance of \eqref{best-response-hat-x}. It follows readily from 
Proposition~\ref{prop:best_response_map_properties} that $\hatx_{(i,\ell)}(\bullet)$ enjoys the following 
properties. 
\begin{lemma}
  \label{lem:loc_best_response_property}
  The best-response  $\hatx_{(i,\ell)} ( \bullet ) $ defined in~\eqref{eq:xhat_definition}  
  satisfies:
  \begin{itemize}
    \item[(1)] $\hatx_{(i,\ell)} ( \bullet ) $ is  $\widehat{L}_{i,\ell}$-Lipschitz continuous on $\KK$;
    \item[(2)] The set of the fixed-points of  
      $\hatx_{(i,:)} ( \bullet ) \triangleq \big(\hatx_{(i,\ell)} ( \bullet )  \big)_{\ell=1}^B$ 
      coincides with the set of stationary solutions of problem~\eqref{eq:problem}.
      \oprocend
  \end{itemize}%
\end{lemma}

\subsection{Convergence of Consensus and Tracking}
\label{sec:consensus_proof}
In this subsection we prove that i) the local  estimates $\bx_{(i,:)}^t$ reach asymptotic consensus 
(cf.~Proposition~\ref{prop:consensus}); and ii) all  $\by_{(i,:)}^t$ are asymptotically consensual while 
tracking the average of the gradients, namely $\frac{1}{N}\sum_{i=1}^N \nabla f_i(\bx_{(:,i)}^t)$ 
(cf.~Proposition~\ref{prop:tracker_convergence}). Note that 
Proposition~\ref{prop:consensus} also proves statement (i) of 
Theorem~\ref{thm:convergence}.

\subsubsection{Achieving consensus}

We begin observing that, for each $\ell\in  \until{B}$,  the block-wise 
$\bx$-update of \algname/ [cf.~\eqref{eq:alg_phi_update}--\eqref{eq:alg_x_update}] 
is an instance of the perturbed push-sum algorithm (\ref{eq:perturbed_push-sum}), 
with $\bepsilon_{i}^t \triangleq \gamma^t \phi_{(i,\ell)}^t \Deltax_{(i,\ell)}^t$ 
and $n=d$. By Lemma~\ref{lem:perturbed_pushsum}(2), it follows that convergence 
of each $\bx_{(i,\ell)}^t$ to the average $\avgdec_\ell^t$ can be readily proven showing 
that each $\Deltax_{(i,\ell)}^t$ is {\it uniformly bounded}. In fact, this together with 
$\gamma^t \downarrow 0$ and $\phi_{(i,\ell)}^t\leq N$, for all $i\in \until{N}$ and 
$t\geq 0$, yields $\lim_{t\to \infty} \|\bepsilon_{i}^t\|=0$ 
(cf.~Proposition~\ref{prop:consensus}).
The following lemma proves that each $\Deltax_{(i,\ell)}^t$ is uniformly bounded.
\begin{lemma}
  \label{lem:bounds_x_y} 
  Consider problem~\eqref{eq:problem} under Assumption~\ref{ass:cost_functions}, 
  \ref{ass:strong_conn}, \ref{ass:col_stoch}, \ref{ass:surrogate}, \ref{ass:step-size}. 
  Let $\{ (\phi_{(i,:)}^t)_{i=1}^N \}_{t\ge 0}$, $\{(\bx_{(i,:)}^t)_{i=1}^N \}_{t\ge 0}$ 
  and $\{(\by_{(i,:)}^t)_{i=1}^N \}_{t\ge 0}$ be generated by \algname/. 
  Then, for all $\ell \in \until{B}$ and $i \in\until{N}$, the following holds:
  \begin{equation}
    \label{eq:y_bound}
    \sup_{t\ge 0} \:
    \Big \|
      \by_{(i,\ell)}^t - \avggrad_\ell^t
    \Big \| < C_1,
  \end{equation}
  and %
  \begin{equation}
    \label{eq:x_bound}
    \sup_{t\ge 0} \:
    \Big \|
      \Deltax_{(i,\ell)}^t
    \Big \| < C_2,
  \end{equation}
  where  $C_1$ and $C_2$ are some  positive, finite scalars.
\end{lemma}
\begin{proof}
  We prove \eqref{eq:y_bound}. %
  Note that the gradient tracking in~\eqref{eq:alg_phi_update}, \eqref{eq:alg_sigma_update} 
  and \eqref{eq:alg_y_update} is an instance of the  perturbed push-sum algorithm (\ref{eq:perturbed_push-sum}), with 
  $
    \bepsilon_i^t \triangleq  \nabla_\ell f_i \big( \bx_{(i,:)}^{t+1} \big)
    \!-\!
    \nabla_\ell f_i \big( \bx_{(i,:)}^{t} \big)$.
   By Lemma~\ref{lem:perturbed_pushsum} 
  (cf. eq.~\eqref{eq:push-sum evolution bound}), we have
  \begin{align*}
  \begin{split}
		& \Big \| 
		  \by_{(i,\ell)}^t - \avggrad_\ell^{t}
		\Big \|
		\\
		&\le
		c_1 (\rho)^{t-1}  + \smallsum_{\tau=1}^{t-1} (\rho)^{t-1-\tau} 
		\smallsum_{i=1}^N \left\|
		  \nabla_\ell f_i (\bx_{(i,:)}^{\tau+1} )
		  \!-\!
      \nabla_\ell f_i (\bx_{(i,:)}^\tau)
		\right\|_1 
		\\
		&	\le	
		c_1 (\rho)^{t-1}  + (2 N\,\sqrt{d}\, B_F)\, \smallsum_{\tau=1}^{t-1} (\rho)^{t-1-\tau}, 
  \end{split}
\end{align*} 
where $B_F \triangleq \sum_{i =1}^N B_{f_i}$ [cf.~Assumption~\ref{ass:cost_functions}(iii)]. 
The above inequality proves~\eqref{eq:y_bound}.

\noindent
We prove now \eqref{eq:x_bound}. 
Consider the case $\ell = \ell_i^t$ [for $\ell \neq \ell_i^t$, $\Deltax_{(i,\ell)}^t=0$, 
trivially implying \eqref{eq:x_bound}].
Invoking Lemma~\ref{lem:strongly_convex_minimization}, with the following 
identifications: $\tildex=\tildex_{(i,\ell)}^t$,
$h(\bullet)=\surr_{i,\ell_i^t}(\bullet; \bx_{(i,:)}^t, \by_{(i,\ell_i^t)}^{t})$, 
$r(\bullet)= r_{\ell_i^t}(\bullet)$, and $\KK= \KK_{\ell_i^t}$, yields   
\begin{align*}
  \left\| \Deltax_{(i,\ell_i^t)}^t \right\|  
  \le \frac{N}{\tau_i }
  \left\| \by_{(i,\ell_i^t)}^t \right\| + \frac{B_{\reg}}{\tau_i},
\end{align*}
where we used the fact that i) 
$\nabla \surr_{i,\ell_i^t}(\bx_{(i,\ell_i^t)}^t; \bx_{(i,:)}^t, \by_{(i,\ell_i^t)}^{t}) = N\by_{(i,\ell)}^t$ 
(cf.~Assumption~\ref{ass:surrogate}); 
and ii) $\|\subgrad r_{\ell_i^t}(\bx_{(i,\ell_i^t)}^t)\|\leq B_{\reg}$ [cf.~Assumption~\ref{ass:cost_functions}(iv)].
By adding and subtracting  $\avggrad_{(i,\ell_i^t)}^t$ in $\|\by_{(i,\ell_i^t)}^t\|$ and using triangle 
inequality we can bound $ \| \Deltax_{(i,\ell_i^t)}^t \|$ as
\begin{align*}
  \| \Deltax_{(i,\ell_i^t)}^t \|  
    &  
    \le 
    \frac{N}{\tau_i }
    \left\| \by_{(i,\ell_i^t)}^t - \avggrad_{(i,\ell_i^t)}^t  \right\| 
    +
     \frac{N}{\tau_i }
    \left\|
      \avggrad_{(i,\ell_i^t)}^t
     \right\|
    + \frac{B_{\reg}}{\tau_i}.
    \\
    & \stackrel{(a)}{\le}
    \frac{N}{\tau_i }
    \smallsum_{\ell=1}^B
    \left\| \by_{(i,\ell)}^t - \avggrad_{(i,\ell)}^t  \right\|
        \\
    & \hspace{1.5em}
    +
        \frac{1}{\tau_i }
    \Big \|
      \smallsum_{i=1}^N \nabla_{\ell_i^t} f_i ( \bx_{(i,:)}^t ) 
    \Big \|
    + \frac{B_{\reg}}{\tau_i}
    \\
    &
    \stackrel{(b)}{\le}
        \dfrac{N}{\tau_i }
    \smallsum_{\ell=1}^B
    \left \| \by_{(i,\ell)}^t - \avggrad_{(i,\ell)}^t\right \| 
    +
        \frac{N}{\tau_i }\cdot B_F
    + \dfrac{B_{\reg}}{\tau_i}
    \\
        &
    \stackrel{(c)}{\le}
        \dfrac{N}{\tau_i } \cdot B \cdot C_1
    +
        \dfrac{N}{\tau_i }\cdot B_F
    + \dfrac{B_{\reg}}{\tau_i} \triangleq C_2 <\infty,
\end{align*}
where in (a) we used~\eqref{eq:gradient_sum_conservation}; 
(b) follows from the boundedness of $\nabla f_i$
[cf.~Assumption~\ref{ass:cost_functions}(iii)]; 
and (c) comes from~\eqref{eq:y_bound}.
\end{proof}

We are now ready to characterize the dynamics of the consensus error, as given below. %
\begin{proposition} 
\label{prop:consensus}
  Consider problem~\eqref{eq:problem}  under Assumptions~\ref{ass:cost_functions}, 
  \ref{ass:strong_conn}, \ref{ass:col_stoch}, \ref{ass:surrogate}, \ref{ass:step-size}. 
  Let $\{(\bx_{(i,:)}^t)_{i=1}^N \}_{t\ge 0}$ and $\{ (\bs_{(i,:)}^t)_{i=1}^N \}_{t\ge 0}$ be 
  generated by \algname/. 
  Then, the decision variables $\bx_{ (i,:) }^t$ are asymptotically consensual to $\avgdec^t$:
  \begin{align}
    \lim_{t\to\infty} \| \bx_{ (i,:) }^t - \avgdec^t \| = 0,
    \label{eq:consenus_agreement}
  \end{align}
  for all $i \in\until{N}$. 
    Furthermore, the following hold:
  \begin{align}
    & \smallsum_{t=0}^\infty \gamma^t \| \bx_{ (i,:) }^t - \avgdec^t \| < \infty,
    \label{eq:summable_err}
    \\
    & \smallsum_{t=0}^\infty \| \bx_{ (i,:) }^t - \avgdec^t \|^2 < \infty.
    \label{eq:squared_summable_err} 
  \end{align}
\end{proposition}
\begin{proof}
  It is sufficient to prove \eqref{eq:consenus_agreement}--\eqref{eq:squared_summable_err} 
  for each block $\ell$. 

  Notice that the  evolution of $\bx_{(:,\ell)}^{t}$ [\eqref{eq:alg_phi_update}--\eqref{eq:alg_x_update}] follows the 
  dynamics of the perturbed push-sum algorithm~\eqref{eq:perturbed_push-sum}, under the following identification:
  $n=d$, $\psi_i^t \triangleq \phi_{(i,\ell)}^t$, 
  $\boldeta_i^t \triangleq \bs_{(i,\ell)}^t$, $\bz_i^t \triangleq \bx_{(i,\ell)}^t$, and 
  $\bepsilon_{i,\ell}^t \triangleq  \gamma^t \phi_{(i,\ell)}^t \Deltax_{(i,\ell)}^t$. 
  By Lemma~\ref{lem:bounds_x_y} [cf.~\eqref{eq:x_bound}] and
  $\gamma^t\downarrow 0$, we infer 
  $\lim_{t\to \infty}\bepsilon_{i,\ell}^t = \gamma^t \phi_{(i,\ell)}^t \Deltax_{(i,\ell)}^t = 0$.
  Invoking Lemma~\ref{lem:perturbed_pushsum}(2), we conclude  
  $\lim_{t\to\infty} \| \bx_{(i,\ell)}^{t} - \avgdec_\ell^{t} \| = 0$,
  which proves~\eqref{eq:consenus_agreement}.
  We prove now~\eqref{eq:summable_err}. Using again the aforementioned connection 
  with the perturbed push-sum algorithm~\eqref{eq:perturbed_push-sum}, we can 
  invoke Lemma~\ref{lem:perturbed_pushsum}(1) [cf.~\eqref{eq:push-sum evolution bound}] 
  and write
  \begin{align} \label{eq:summable_err_proof}
  \begin{split}
    & \smallsum_{t=0}^{\infty}
    \gamma^{t+1}\norm{\bx_{\ell}^{t+1} - \avgdec_\ell^{t+1}}
    \\
    & \le
    \smallsum_{t=0}^{\infty}\gamma^{t+1} 
     \Big( 
       c_1\, (\rho)^{t} \! + c_2\, \smallsum_{\tau=1}^{t} (\rho)^{t-\tau} \gamma^\tau \| \Deltax_{(:,\ell)}^t \|_1 
     \Big) 
     \\
      & 	\stackrel{(a) }{\le}
	\smallsum_{t=0}^{\infty} \gamma^{t+1} \!
     \Big( 
	 c_1 \,(\rho)^{t} + c_3 \,\smallsum_{\tau=1}^{t} (\rho)^{t-\tau} \gamma^\tau \,
	 \Big)
     \stackrel{(b) }{<}\infty,
    \end{split}
  \end{align}
  for some finite, positive scalars $c_1, c_2$, and $c_3$, where (a) follows from the boundedness 
  of $\| \Deltax_{(:,:)}^t \|_1 $ [cf. Lemma~\ref{lem:bounds_x_y}]; and (b) is 
  due to~\cite[Lemma~7]{nedic2010constrained}.

Finally, to prove~\eqref{eq:squared_summable_err}, we use the same bound of $\norm{\bx_\ell^{t+1} - \avgdec_\ell^{t+1}}$ as in~\eqref{eq:summable_err_proof}, and write
\begin{align*}
\begin{split}
	& \smallsum_{t=0}^{\infty}\norm{\bx_\ell^{t+1} - \avgdec_\ell^{t+1}}^2
	\\	
	&\leq  \smallsum_{t=0}^{\infty} \Big(
	c_1^2 (\rho)^{2t} +
	c_3^2 
	\smallsum_{\tau=0}^{t} 
	\smallsum_{s=0}^{t}\gamma^\tau \gamma^s (\rho)^{t-\tau} (\rho)^{t-s}
	\\
	& \hspace{3.0cm}
	+ 2 c_1c_3 
	\smallsum_{\tau=0}^{t}\gamma^\tau (\rho)^{t-\tau} (\rho)^{t} \Big)
	\stackrel{(a)}{<}
	\infty,
\end{split}
\end{align*}
where (a) follows from~\cite[Lemma~7]{dilorenzo2016next}.%
\end{proof}

\subsubsection{Asymptotic tracking} 
We conclude this section studying the dynamics of the gradient tracking scheme.%
\begin{proposition}
\label{prop:tracker_convergence}
  Consider problem~\eqref{eq:problem} under Assumptions~\ref{ass:cost_functions}, 
  \ref{ass:strong_conn}, \ref{ass:col_stoch}, \ref{ass:surrogate}, \ref{ass:step-size}. 
  Let $\{ (\by_{(i,:)}^t)_{i=1}^N\}_{t\geq 0}$ be the sequence 
  generated by \algname/.
  Then, $\by_{(i,:)}^t$ tracks the average of the gradients 
  $\sum_{j=1}^N \nabla_\ell f_j (\bx_{(j,:)}^t )$ asymptotically:
  \begin{align}
  \label{eq:tracker_convergence}
    \lim_{t\to \infty} 
    \Big \| \by_{(i,:) }^t - \frac{1}{N} \smallsum_{j=1}^N \nabla f_j (\bx_{(j,:)}^t ) \Big\| 
    = 0,
  \end{align}
  for all $i \in\until{N}$. 
  Furthermore, the following holds:
  \begin{align}
    \label{eq:tracker_summable_err}
    \smallsum_{t=0}^\infty \gamma^t  \Big \| \by_{(i,:) }^t 
    - \dfrac{1}{N} \smallsum_{j=1}^N \nabla f_j (\bx_{(j,:)}^t ) \Big\| < \infty.
  \end{align}%
\end{proposition}%
\begin{proof}
  It is sufficient to prove~\eqref{eq:tracker_convergence} and~\eqref{eq:tracker_summable_err} 
  for each block $\ell$. 
  Notice that the gradient tracking scheme given by \eqref{eq:alg_phi_update}, 
  \eqref{eq:alg_sigma_update} and \eqref{eq:alg_y_update} is an instance of the 
  perturbed push-sum consensus~\eqref{eq:perturbed_push-sum}, with the 
  identifications: $n=d$, $\psi_i^t \triangleq \phi_{(i,\ell)}^t$, 
  $\boldeta_i^t \triangleq \bsigma_{(i,\ell)}^t$, 
  $\bz_i^t \triangleq \by_{(i,\ell)}^t$, and 
  $\bepsilon_{i,\ell}^t \triangleq \nabla_\ell f_i (\bx_{(i,:)}^{t+1} )
  \!-\!
  \nabla_\ell f_i (\bx_{(i,:)}^t)$.
  Therefore, \eqref{eq:tracker_convergence} follows readily from 
  Lemma~\ref{lem:perturbed_pushsum}(2) and \eqref{eq:gradient_sum_conservation}, 
  once we have shown
  $\lim_{t \to \infty} \|\bepsilon_{i,\ell}^t\|=0 $, as proven next.

  Since each $\nabla_\ell f_i$  is Lipschitz continuous [cf.~Assumption~\ref{ass:surrogate}(iii)], it suffices to prove $\lim_{t\to\infty} \| \bx_{(i,:)}^{t+1} - \bx_{(i,:)}^{t} \big \| = 0$.
We have: 
  \begin{align}
    \begin{split}
    \left\| \bx_{(i,:)}^{t+1} - \bx_{(i,:)}^{t}   \right\|
     \stackrel{(a)}{\le} & \,
   \left \| \bx_{(i,:)}^{t+1} - \avgdec^{t+1} \right\|
    +    
  \left  \| \bx_{(i,:)}^{t} - \avgdec^{t} \right \|    
    \\
     & + \frac{1}{N}  \smallsum_{\ell=1}^B\smallsum_{i=1}^N \left\|  \gamma^t
     \phi_{(i,\ell)}^t \Deltax_{(i,\ell)}^t  \right \|
    \\
     \stackrel{(b)}{\le} &\,
  \left  \| \bx_{(i,:)}^{t+1} - \avgdec^{t+1} \right \|
    +    
  \left  \| \bx_{(i,:)}^{t} - \avgdec^{t} \right \|
    \\
    & +
    \gamma^t 
    \smallsum_{\ell=1}^B
    \smallsum_{i=1}^N  \left\| \Deltax_{(i,\ell)}^t\right\|,
    \end{split}
    \label{eq:bounded_average_evolution}
  \end{align}
  where  in (a) we used \eqref{eq:avg_dec_block_evolution} while  (b) follows from $\phi_{(i,\ell)}^t \le N$.
 The desired result,   $\lim_{t\to\infty} \| \bx_{(i,:)}^{t+1} - \bx_{(i,:)}^{t} \big \| = 0$, follows readily from \eqref{eq:bounded_average_evolution},  Proposition~\ref{prop:consensus} [cf. eq.~\eqref{eq:consenus_agreement}],  
  Lemma~\ref{lem:bounds_x_y}(2), and   $\gamma^t\downarrow 0$ [cf.~Assumption~\ref{ass:step-size}].	
	
We prove now~\eqref{eq:tracker_summable_err}. Invoking Lemma~\ref{lem:perturbed_pushsum}(2), 
we can write
\begin{align*}
  & \smallsum_{t=0}^\infty \gamma^{t+1}
  \Big \| 
      \by_{(i,\ell) }^{t+1} \!-\! \dfrac{1}{N} \smallsum_{j=1}^N \nabla_\ell f_j (\bx_{(j,:)}^{t+1} ) 
  \Big \| 
  \\
  & =
   \smallsum_{t=0}^\infty \gamma^{t+1}
   \left \|
      \by_{(i,\ell) }^{t+1} \!-\! \avggrad_\ell^{t+1}
    \right \|
  \\
  & \le   
    \smallsum_{t=0}^\infty \gamma^{t+1}\Big(c_1 ( \rho)^{t} \\
     & \hspace{3em}+ c_2 \smallsum_{\tau=1}^{t} (\rho)^{t-\tau} 
		\smallsum_{i=1}^N 
		\left \|
		  \nabla_\ell f_i (\bx_{(i,:)}^{\tau+1} )
		  \!-\!
      \nabla_\ell f_i (\bx_{(i,:)}^\tau)
    \right \|\Big)
    \\
		&\le  
		\smallsum_{t=0}^\infty \gamma^{t+1} \Big( c_1 ( \rho)^{t} 
		+ c_4 \smallsum_{\tau=1}^{t} (\rho)^{t-\tau} 
		\smallsum_{i=1}^N
		\big \| 
		  \bx_{(i,:)}^{\tau+1} - \bx_{(i,:)}^{\tau}
		\big \|\Big)
  \\
  & \stackrel{\eqref{eq:bounded_average_evolution}}{\le}   
  c_1 \smallsum_{t=0}^\infty \gamma^{t+1}( \rho)^{t} 
  + 
  c_4 \smallsum_{t=0}^\infty \gamma^{t+1}\smallsum_{\tau=1}^{t} (\rho)^{t-\tau} 
  \smallsum_{i=1}^N
	  \| \bx_{(i,:)}^{\tau+1} - \avgdec^{\tau+1} \big \|
  \\
  &
  +    
  c_4 \smallsum_{t=0}^\infty \gamma^{t+1}\smallsum_{\tau=1}^{t} (\rho)^{t-\tau} 
  \smallsum_{i=1}^N\| \bx_{(i,:)}^{\tau} - \avgdec^{\tau} \big \|
  \\     
  &     
  +
  c_5 \smallsum_{t=0}^\infty \gamma^{t+1}\smallsum_{\tau=1}^{t} (\rho)^{t-\tau} 
  \gamma^\tau 
  \stackrel{(a)}{<} \infty,
	\end{align*}
  for some positive, finite scalars $c_4$ and $c_5$, where (a) follows 
  from~\cite[Lemma~7]{dilorenzo2016next}.   
\end{proof}

\subsection{Lyapunov Function and its Descent Property}
\label{sec:opt}
We begin introducing  the following lemma  that is instrumental for the rest of the proof.

\begin{lemma}
\label{prop:descent_on_reg}
  Consider problem~\eqref{eq:problem}  under Assumptions~\ref{ass:cost_functions}, 
  \ref{ass:strong_conn}, \ref{ass:col_stoch}, \ref{ass:surrogate}, \ref{ass:step-size}; and let  
  $\{ \phi_{(i,:)}^t \}_{t\ge 0}$ and $\{ \bx_{(i,:)}^t \}_{t\ge 0}$ be 
  the sequences generated by \algname/. Then,  for all $\ell \in \until{B}$, 
  it holds
  \begin{align*}
\begin{split}
  & 
  \smallsum_{i=1}^N \phi_{(i,\ell)}^{t+1}\,
  \reg_\ell (\bx_{(i,\ell)}^{t+1}) 
  - 
 \smallsum_{i=1}^N \phi_{(i,\ell)}^{t}\,
  \reg_\ell (\bx_{(i,\ell)}^{t}) 
  \\
  &
  \le
  \gamma^t
  \dfrac{1}{N}
  \smallsum_{i=1}^N
  \phi_{(i,\ell)}^{t} 
  \Big(
  \reg_{\ell} (\tildex_{(i,\ell ) }^t) 
  -
  \reg_{\ell} (\bx_{(i,\ell ) }^t)
  \Big).
\end{split}%
\end{align*}
\end{lemma}%
\begin{proof}
  The proof follows readily from the convexity of $r_{\ell}$ and the column 
  stochasticity of $\mathbf{A}_\ell^t$.
\end{proof}

We are now ready to introduce our Lyapunov-like function: given $\avgdec_\ell^t$, 
$(\bx_{(i,\ell)}^t)_{i=1}^N$, and  $(\phi_{(i,\ell)}^t)_{i=1}^N$, define (we omit the 
dependence on the algorithm variables for notational simplicity)
\begin{equation*} %
  V^t\triangleq 	\smallsum_{i=1}^N f_i ( \avgdec^{t+1} ) +  \smallsum_{\ell=1}^B
  \smallsum_{i=1}^N
  \phi_{(i,\ell)}^t\, \reg_{\ell} ( \bx_{(i,\ell)}^t).
\end{equation*}
The descent properties of the above function along the trajectory of the algorithm are studied in the following proposition. 
\begin{proposition}
\label{prop:cost_descent}
  Consider problem~\eqref{eq:problem},  under Assumptions~\ref{ass:cost_functions}, 
  \ref{ass:strong_conn}, \ref{ass:col_stoch}, \ref{ass:surrogate}, \ref{ass:step-size}; and let
  $\{(\phi_{(i,:)}^t)_{i=1}^N \}_{t\ge 0}$, $\{ \avgdec^t\}_{t \geq 0}$,
   and $\{(\bx_{(i,:)}^t)_{i=1}^N \}_{t\ge 0}$  be 
 the sequences  generated by \algname/.  
  Then  $\{V^t\}_{t\geq 0}$ satisfies:
  \begin{align}
  \label{eq:cost_descent_telescopic}
  \begin{split}
    & 
   V^{t+1}\leq    V^t
       -c_7
    \smallsum_{\ell=1}^B
    \smallsum_{i=1}^N
    \gamma^{t} 
    \| \Deltax_{ (i,\ell) }^{t}  \|^2 
    +
    P^{t},
  \end{split}
  \end{align}
  with $\sum_{t=0}^\infty P^t<\infty$,   where $P^t$ is defined as
  \begin{align*}
  \begin{split}
    P^t & \triangleq
  c_8\,
  \gamma^t\, 
  \smallsum_{\ell=1}^B 
  \smallsum_{i=1}^N 
  \Big\| 
  \dfrac{1}{N}
  \smallsum_{j=1}^N
    \nabla_{\ell} 
    f_{j} ( \avgdec^{t} ) - \by_{(i,\ell)}^{t} 
  \Big\|
  +
  c_6\, (\gamma^t)^2,
  \end{split}
  \end{align*}
and $c_6, c_7$, and $c_8$ are some positive, finite scalars.
\end{proposition}
\begin{proof}
Applying the descent lemma to~\eqref{eq:avg_dec_block_evolution}, with  $L = \sum_{i=1}^N L_i $, yields
\begin{align*}
  & \smallsum_{i=1}^N f_i ( \avgdec^{t+1} )
  \\  
  &
  \le
  \smallsum_{i=1}^N f_i ( \avgdec^{t} )
  + 
  \Big( \smallsum_{j=1}^N \nabla f_j ( \avgdec^{t} ) \Big) ^{\!\top }\!
  \Big( \avgdec^{t+1} \! -\! \avgdec^{t} \Big)
  +
  \dfrac{ L }{2}\, \| \avgdec^{t+1} \!-\! \avgdec^{t} \|^2
  \\
  &\le 
  \smallsum_{i=1}^N f_i ( \avgdec^{t} ) 
  + 
  \smallsum_{\ell=1}^B
  \Big(
  \smallsum_{j=1}^N \nabla_{\ell} f_{j} ( \avgdec^{t} )\Big) ^\top 
  \Big( 
  \dfrac{\gamma^t}{N} 
  \smallsum_{i=1}^N
  \phi_{ (i, \ell) }^{t} 
  \Deltax_{ (i,\ell) }^{t}
  \Big)
  \\
  &
 \quad  +
  \dfrac{ L }{2} 
  \smallsum_{\ell=1}^B
  \Big\|
    \dfrac{1}{N}\,
    \gamma^t
    \smallsum_{i=1}^N 
    \phi_{ (i,\ell ) }^{t} \,
    \Deltax_{ (i,\ell ) }^{t} 
  \Big\|^2
  \\
  &\stackrel{(a)}{\le} \smallsum_{i=1}^N f_i ( \avgdec^{t} ) 
  +
  \gamma^t 
  \smallsum_{\ell=1}^B 
  \smallsum_{i=1}^N 
  \phi_{(i,\ell)}^t
  \big( \by_{(i,\ell)}^{t} \big)^\top 
  \Deltax_{ (i,\ell) }^{t}\\   %
   &
  \quad +
  \gamma^t
  \smallsum_{\ell=1}^B 
  \smallsum_{i=1}^N 
  \phi_{(i,\ell)}^t
  \Big( 
  \dfrac{1}{N}
    \smallsum_{j=1}^N
    \nabla_{\ell} 
    f_{j} ( \avgdec^{t} ) - \by_{(i,\ell)}^{t} 
  \Big)^\top 
  \Deltax_{ (i,\ell) }^{t}
  \\
  & 
  \quad +
  (\gamma^t)^2 
  \dfrac{ L }{2}\,
  \smallsum_{\ell=1}^B
  \smallsum_{i=1}^N
  \dfrac{\phi_{(i,\ell)}^t}{N}
  \,\| \Deltax_{ (i,\ell) }^{t} \|^2,
  \\
  &
  \stackrel{(b)}{\le}  \smallsum_{i=1}^N f_i ( \avgdec^{t} ) 
  -
  \gamma^t \tau
  \smallsum_{\ell=1}^B
  \smallsum_{i=1}^N
  \phi_{(i,\ell)}^t
  \| \Deltax_{ (i,\ell) }^{t} \|^2
  \\
  &\quad 
  - \gamma^t
  \smallsum_{\ell=1}^B
  \smallsum_{i=1}^N
  \phi_{(i,\ell)}^t
    \big( \reg_{\ell}( \tildex_{(i,\ell )}^t) - \reg_{\ell} ( \bx_{(i,\ell)}^t) \big)
  \\
  &\quad 
  +
  \gamma^t
  \smallsum_{\ell=1}^B 
  \smallsum_{i=1}^N 
  \phi_{(i,\ell)}^t
  \Big \| 
  \dfrac{1}{N}
  \smallsum_{j=1}^N
    \nabla_{\ell} 
    f_{j} ( \avgdec^{t} ) - \by_{(i,\ell)}^{t} 
  \Big \|\,
  \Big \|
    \Deltax_{ (i,\ell) }^{t} 
  \Big \|
  \\
  & \quad 
  +
  c_6 (\gamma^t)^2,
\end{align*}
where in (a) we added and subtracted $\gamma^t 
\sum_{i=1}^N $ $
\sum_{\ell=1}^B
\phi_{(i,\ell )}^t (\by_{(i,\ell )}^t)^\top \Deltax_{ (i,\ell ) }^{t}$; 
and in (b) we used Lemma~\ref{lem:gradient_related_condition} [cf.~\eqref{eq:sufficient_descent}], 
Lemma~\ref{lem:bounds_x_y} [cf.~\eqref{eq:x_bound}], 
we defined  {$\tau=\min_i \tau_i$}, 
and $c_6$ is some positive, finite scalar. %

Combining now the above chain of inequalities with  Lemma~\ref{prop:descent_on_reg}  and using Lemma~\ref{lem:perturbed_pushsum}(3), we can write %
\begin{align*}
  & 
 V^{t+1}\leq V^t   
  - c_7\,
  \gamma^t \,
  \smallsum_{\ell=1}^B
  \smallsum_{i=1}^N
  \| \Deltax_{ (i,\ell) }^{t} \|^2
  \\
  &
 \quad  +
  \underbrace{c_8
  \gamma^t
  \smallsum_{\ell=1}^B 
  \smallsum_{i=1}^N 
  \Big\| 
  \dfrac{1}{N}
  \smallsum_{j=1}^N
    \nabla_{\ell} 
    f_{j} ( \avgdec^{t} ) - \by_{(i,\ell)}^{t} 
  \Big\|
  +
  c_6\, (\gamma^t)^2}_{P^t},
\end{align*}
where $c_7$, and $c_8$ are some positive, finite scalars. 

To conclude the proof, we show next that $P^t$ is summable. Since 
$\sum_{t=0}^\infty (\gamma^t)^2<\infty$ (cf.~Assumption~\ref{ass:step-size}), 
it is sufficient to prove that the first term of $P^t$ is summable, as shown below:
\begin{align*}
  &
  \lim_{k\to \infty}
  \smallsum_{t=0}^k\gamma^t
  \smallsum_{\ell=1}^B 
  \smallsum_{i=1}^N 
  \Big \| 
  \dfrac{1}{N}
  \smallsum_{j=1}^N
    \nabla_{\ell} 
    f_{j} ( \avgdec^{t} ) - \by_{(i,\ell)}^{t} 
  \Big\|
  \\ 
  &
  \stackrel{(a)}{\le}
 \lim_{k\to \infty} 
 \smallsum_{t=0}^k \gamma^t
  \smallsum_{\ell=1}^B 
  \smallsum_{i=1}^N 
  \Big \|
  \dfrac{1}{N}
  \smallsum_{j=1}^N
    \nabla_{\ell} 
    f_{j} ( \bx_{(j,:)}^t ) - \by_{(i,\ell)}^{t} 
  \Big \|
  \\
  & \quad +
  c_9\, \lim_{k\to \infty} \smallsum_{t=0}^k \gamma^t
  \smallsum_{\ell=1}^B 
  \smallsum_{i=1}^N 
 \big \|
 \bx_{(i,:)}^t 
  - 
 \avgdec^{t} 
  \big \| \stackrel{(b)}{<} \infty,
\end{align*}
where in (a) we used the Lipschitz continuity of $\nabla f_i$; 
(b) follows from Prop.~\ref{prop:consensus} 
and~\ref{prop:tracker_convergence} with $c_9$ positive scalar.
\end{proof}

\subsection{Asymptotic Convergence of  $\{\bar{\bs}^t\}_{t\geq 0}$}
\label{sec:convergence_s_bar}

Since $U$ is coercive and $\sum_{t=0}^\infty P^{t}<\infty$, \eqref{eq:cost_descent_telescopic} 
implies that i) $\{V^t\}_{\geq 0}$ is convergent; and ii)  and $\{\bar{\bs}^t\}_{t\geq 0}$ is bounded. Therefore, it must be 
\begin{align}
  \label{eq:summable_dec}
  \smallsum_{t=0}^\infty 
   \smallsum_{i=1}^N  \smallsum_{\ell=1}^B
  \gamma^{t}
  \| \Deltax_{(i,\ell) }^{ t} \|^2 
  <\infty.
\end{align}

Recall that agents select their blocks to update according to an essential cyclic rule 
[cf. Assumption~\ref{ass:block_selection}].
This means that in any time window $[t,t+T-1]$, 
with $T>0$ defined in Proposition~\ref{prop:induced_graph_connectivity}, 
any agent $i$ selects all of its blocks at least once.
Denote by $t + s_i^t(\ell)$ the last time agent $i$ selects block $\ell$ in the 
time window $[t,t+T-1]$; notice that such a $s_i^t(\ell)$ is always well-defined 
and $s_i^t(\ell)\in [0, T-1]$. Finally, let 
\begin{equation} \label{eq:Delta^t}
  \bDelta^t 
  \triangleq 
  \smallsum_{i=1}^N 
  \smallsum_{\ell=1}^{B}
  \| \Deltax_{(i, \ell)}^{t+s_i^t(\ell)} \|.
\end{equation} 
The above quantity will play a key role to prove (subsequence) convergence of 
$\{\bar{\bs}^t\}_{t\geq 0}$. We organize the rest of the proof in the following steps:
\begin{itemize}
	\item \textbf{Step~1:} We  prove $\lim_{t\to \infty} \bDelta^t = 0$, by showing that, first,
	$\liminf_{t\to \infty} \bDelta^t = 0$  [Step~1(a)], and, second, $\limsup_{t\to \infty} \bDelta^t = 0$ [Step~1(b)];\medskip 
	\item\textbf{Step~2:}  Using results in Step~1, we prove that every limit point of $\{\bar{\bs}^t\}_{t\geq 0}$ is a stationary solution of problem~\eqref{eq:problem}.\smallskip 
\end{itemize}  
  
  \noindent  {{\bf Step~1(a)} -- $\liminf_{t\to \infty}\bDelta^t=0$}. For all $t\geq T-1$, we have %
\begin{align}\label{eq:summable_dec_window}
\begin{split}
 &T \cdot
  \smallsum_{\tau=0}^t 
  \smallsum_{\ell=1}^B
  \smallsum_{i=1}^N 
  \gamma^{\tau}
  \| \Deltax_{(i,\ell) }^{ \tau} \|^2\smallskip 
  \\
  &\qquad \ge
   \smallsum_{\tau=0}^{t-T+1} 
  \smallsum_{s =0}^{T-1}
  \smallsum_{\ell=1}^B
  \smallsum_{i=1}^N 
  \gamma^{\tau +s}
  \| \Deltax_{(i,\ell) }^{ \tau + s} \|^2 \smallskip 
  \\
 & \qquad \stackrel{(a) }{\ge} 
     \smallsum_{\tau=0}^{t-T+1} 
       \gamma^{\tau +T-1}
  \smallsum_{s =0}^{T-1}
  \smallsum_{\ell=1}^B
  \smallsum_{i=1}^N 
  \| \Deltax_{(i,\ell) }^{ \tau + s} \|^2,
  \end{split}
\end{align}
where $(a)$ follows from Assumption~\ref{ass:step-size}(i). Using  \eqref{eq:summable_dec} and $\sum_{t=0}^\infty \gamma^t = \infty$, we deduce    $$    \liminf_{t\to\infty} 
  \smallsum_{s =0}^{T-1}
  \smallsum_{\ell=1}^B
  \smallsum_{i=1}^N 
  \| \Deltax_{(i,\ell) }^{ t + s} \|=0,$$
  which leads to 
 $$
 0=    
    \liminf_{t\to\infty}    
  \smallsum_{\ell=1}^B
  \smallsum_{i=1}^N 
  \smallsum_{s =0}^{T-1}
  \| \Deltax_{(i,\ell) }^{t + s} \| \geq  \displaystyle \liminf_{t\to\infty}  
 \bDelta^t.	 
$$

 \noindent {{\bf Step~1(b)} -- $\limsup_{t\to \infty}\bDelta^t=0$}. We begin  stating  the following lemma, which proves that the best-response maps $\widetilde{\bx}_{(i,\ell_i^t)}$ 
[cf.~\eqref{eq:alg_local_min}] and $\widehat{\bx}_{(i,\ell_i^t)}$ [cf.~\eqref{eq:xhat_definition}],   are asymptotically consistent along the trajectory of the algorithm.  
\begin{lemma} 
\label{lem:best_response_consistency}
In the setting of \algname/, the best-response maps $\widehat{\bx}_{(i,\ell_i^t)}$ and $\widetilde{\bx}_{(i,\ell_i^t)}^t$  satisfy
  \begin{align}
    \lim_{t\to\infty}
   \left \| 
      \hatx_{(i,\ell_i^t)} \big( \bx_{(i,:)}^t \big) - \tildex_{(i,\ell_i^t)}^t
    \right\|
    = 0, \hspace{0.2cm} \forall \, i\in \until{N}.
    \label{eq:best_response_consistency}
  \end{align}
\end{lemma} 
\begin{proof}
  We use the shorthand   $\hatx_{(i,\ell_i^t)}^t$ for $\hatx_{(i,\ell_i^t)} \big ( \bx_{(i,:)}^t \big )$. Invoking the 
  optimality conditions of $\hatx_{(i,\ell_i^t)} ( \bx_{(i,:)}^t ) $ 
  and $\tildex_{(i,\ell_i^t)}^{t}$ yields
  \begin{align}
  \label{eq:optimality_xhat}
  \begin{split}
    & \big( \tildex_{(i,\ell_i^t)}^t - \hatx_{(i,\ell_i^t)}^t \big)^\top \times
    \\
    & 
    \Big ( 
      \nabla_{\ell_i^t} 
      \surr_{i,\ell_i^t} 
      \big (
        \hatx_{(i,\ell_i^t)}^t ; \bx_{(i,:)}^t , {\textstyle \frac{1}{N}} \smallsum_{j=1}^N \nabla_{\ell_i^t} 
        f_j (\bx_{(i,:)}^t)
      \big )
    \\
    & 
    +
    \subgrad \reg_{\ell_i^t}
    \big(
      \hatx_{(i,\ell_i^t)}^t
    \big)
    \Big )
    \ge 0,
  \end{split}
  \end{align}
  and
  \begin{align}
  \label{eq:optimality_xtilde}
  \begin{split}
    & 
    \big( \hatx_{(i,\ell_i^t)}^t - \tildex_{(i,\ell_i^t)}^t \big)^\top  \times
    \\
    &
    \Big ( 
    \nabla_{\ell_i^t} \surr_{i,\ell_i^t} (\tildex_{(i,\ell_i^t)}^t; \bx_{(i,:)}^t, \by_{(i,\ell_i^t)}^t)
    +
    \subgrad \reg_{\ell_i^t}
    \big(
      \tildex_{(i,\ell_i^t)}^{t }
    \big)
    \Big )
    \ge 0.
  \end{split}
  \end{align}
  Adding the two inequalities~\eqref{eq:optimality_xhat} and~\eqref{eq:optimality_xtilde}
  and using the strong convexity of $\tf_i (\bullet; \bx_{(i,:)}^t )$ as well as
  the convexity of $\reg_{\ell_i^t}$, yields
  \begin{align*}
  & \Big \| \tildex_{(i,\ell_i^t)}^t - \hatx_{(i,\ell_i^t)}^t \Big \| 
    \le   
    \dfrac{1}{\tau_i} 
    \Big \| 
       \dfrac{1}{N} \smallsum_{j=1}^N \nabla_{\ell_i^t} f_j(\bx_{(i,:)}^t) 
      -
      \by_{(i,\ell_i^t)}^t 
    \Big \|
    \\
    &\quad \le 
    \dfrac{1}{\tau_i} 
    \Big\| 
      \dfrac{1}{N} \smallsum_{j=1}^N \nabla_{\ell_i^t} f_j(\bx_{(i,:)}^t) 
      -
      \dfrac{1}{N} \smallsum_{j=1}^N \nabla_{\ell_i^t} f_j(\bx_{(j,:)}^t) 
    \Big\|
    \\
    & \qquad +
        \dfrac{1}{\tau_i} 
    \Big\| 
       \dfrac{1}{N} \smallsum_{j=1}^N \nabla_{\ell_i^t} f_j(\bx_{(j,:)}^t) 
      -
      \by_{(i,\ell_i^t)}^t 
    \Big\|
    \\
    &\quad \le 
    \dfrac{1}{\tau_i\, N} \smallsum_{j=1}^N L_j\, \left\| \bx_{(i,:)}^t - \bx_{(j,:)}^t \right\|
    \\
    &
    \qquad +
        \dfrac{1}{\tau_i} 
    \Big\| 
       \dfrac{1}{N} \smallsum_{j=1}^N \nabla_{\ell_i^t} f_j(\bx_{(j,:)}^t) 
      -
      \by_{(i,\ell_i^t)}^t 
    \Big\|
        \\ 
    &\quad \le
    \dfrac{1}{\tau_i N} 
    \smallsum_{j=1}^N L_j \Big( \Big \| \bx_{(i,:)}^t - \avgdec^t \| + \| \avgdec^t - \bx_{(j,:)}^t \Big\| \Big)
    \\
    &
    \qquad+  
        \dfrac{1}{\tau_i} 
    \Big \| 
      \dfrac{1}{N} \smallsum_{j=1}^N \nabla_{\ell_i^t} f_j(\bx_{(j,:)}^t) 
      -
      \by_{(i,\ell_i^t)}^t 
    \Big \|.
  \end{align*}
  Finally, by noticing that
  \begin{align*}
  \begin{split}
    & 
    \limsup_{t\to\infty} 
    \Big \|
       \dfrac{1}{N} \smallsum_{i=1}^N \nabla_{\ell_i^t} f_i(\bx_{(i,:)}^t) 
      -
      \by_{(i,\ell_i^t)}^t     
    \Big \|
    \\ 
    & 
    \le
    \limsup_{t\to\infty} 
    \smallsum_{\ell=1}^{B}
    \Big \|
       \dfrac{1}{N} \smallsum_{i=1}^N \nabla_{\ell} f_i(\bx_{(i,:)}^t) 
      -
      \by_{(i,\ell)}^t 
    \Big \|,
  \end{split}
  \end{align*}
  and invoking Propositions~\ref{prop:consensus} and~\ref{prop:tracker_convergence}, 
  we obtain the desired result
 $   \limsup_{t\to \infty} \big \| \hatx_{(i,\ell_i^t)}- \tildex_{(i,\ell_i^t)}^t \big \| = 0$.
\end{proof}

Now we prove by contradiction that $\limsup_{t\to\infty } \bDelta^{t} = 0$.
Suppose
$\limsup_{t\to\infty } \bDelta^{t} >0$.
Since $\liminf_{t\to\infty} \bDelta^{t} =0$,   there exists a $\delta > 0$ such that $\bDelta^{t} < \delta$ for infinitely many $t$
and also $\bDelta^{t} > 2\delta$ for infinitely many $t$.
Therefore, one can always find an infinite set of indices, say $\TT$, having the following property:
for any $t \in \TT$, there exists an integer $\theta_t > t$ such that
\begin{align}\label{eq:limsup_cond}
\begin{split}
  & \bDelta^{t} \le \delta, \: \: \bDelta^{\theta_t} \ge 2\delta,
  \\
  & \delta < \bDelta^\tau < 2\delta, \hspace{1cm} t  < \tau < \theta_t.
\end{split}
\end{align}

Therefore, for all $t\in\TT$, we have
\begin{align}
\begin{split}
  \delta  
  & \le
  \bDelta^{\theta_t} - \bDelta^{t}
  \\  
  & = 
  \smallsum_{i=1}^N 
  \smallsum_{\ell=1}^{B} \!
  \Big( 
  \Big \|
  \tildex_{(i,\ell)}^{\theta_t+s_i^{\theta_t}(\ell)} - 
  \bx_{(i,\ell)}^{\theta_t+s_i^{\theta_t}(\ell)} \Big \| 
  \! - \!
  \Big \|
   \tildex_{(i,\ell)}^{t+s_i^t(\ell)} - 
  \bx_{(i,\ell)}^{t+s_i^t(\ell)} 
  \Big\| 
  \Big)
  \\  
  & 
  \le
  \smallsum_{i=1}^N 
  \smallsum_{\ell=1}^{B} \! 
  \Big ( 
  \Big \|
	  \tildex_{(i,\ell)}^{\theta_t+s_i^{\theta_t}(\ell)} -
	  \tildex_{(i,\ell)}^{t+s_i^t(\ell)} 
  \Big \| 
  \! + \!
  \Big \| 
    \bx_{(i,\ell)}^{\theta_t+s_i^{\theta_t}(\ell)} 
    -
    \bx_{(i,\ell)}^{t+s_i^t(\ell)} 
  \Big \| \Big )
  \\
  &
  \leq  
  \smallsum_{i=1}^N 
  \smallsum_{\ell=1}^{B} \!
  \Big (
  \Big \|
  \tildex_{(i,\ell)}^{\theta_t+s_i^{\theta_t}(\ell)} -
  \hatx_{(i,\ell)}(\bx_{(i,:)}^{\theta_t+s_i^{\theta_t}(\ell)}) \Big \|
  \\
  &  \hspace{1.3cm} +
  \Big \| 
    \hatx_{(i,\ell)}(\bx_{(i,:)}^{\theta_t+s_i^{\theta_t}(\ell)}) -
    \hatx_{(i,\ell)}(\bx_{(i,:)}^{t+s_i^t(\ell)})  \Big \|
  \\
  &  \hspace{1.3cm} +
  \Big \|
    \hatx_{(i,\ell)}(\bx_{(i,:)}^{t+s_i^t(\ell)}) -
    \tildex_{(i,\ell)}^{t+s_i^t(\ell)} 
   \Big \|
  \\
  &  \hspace{1.3cm} +
   \Big \| 
    \bx_{(i,\ell)}^{\theta_t+s_i^{\theta_t}(\ell)}  - 
    \bx_{(i,\ell)}^{t+s_i^t(\ell)} 
 \Big \| \Big )
  \\
   & 
   \leq 
   (1+ \widehat{L})\smallsum_{i=1}^N 
   \smallsum_{\ell=1}^{B}
   \Big \| 
    \bx_{(i,:)}^{\theta_t+s_i^{\theta_t}(\ell)} -
    \bx_{(i,:)}^{t+s_i^t(\ell)} 
    \Big \| 
    + e_1^t,
\end{split}
\label{eq:limsup_bound_1}
\end{align}
where in the last inequality, we used the Lipschitz continuity 
of $\hatx_{(i,\ell)}(\bullet)$ [cf. Lemma.~\ref{lem:loc_best_response_property}], 
with $\widehat{L} \triangleq \max_{i} \max_{\ell} \widehat{L}_{i,\ell}$, and 
\begin{align}
\begin{split}
  &   
  e_1^t \triangleq
  \smallsum_{i=1}^N
  \smallsum_{\ell=1}^{B} 
  \bigg (  
  \Big \|
    \hatx_{(i,\ell)}(\bx_{(i,:)}^{\theta_t+s_i^{\theta_t}(\ell)}) - 
    \tildex_{(i,\ell)}^{\theta_t+s_i^{\theta_t}(\ell)}
    \Big \|
  \\
 &
 \hspace{2.8cm}
 +   
 \Big \|
    \hatx_{(i,\ell)}(\bx_{(i,:)}^{t+s_i^t(\ell)}) - 
    \tildex_{(i,\ell)}^{t+s_i^t(\ell)} 
  \Big \| \bigg ).
\end{split}
\label{eq:e1_def}
\end{align}

Adding and subtracting  {$\avgdec^{\theta_t+s_i^{\theta_t}(\ell)}$ and $\avgdec^{t+s_i^{t}(\ell)}$} in 
the first term of the last inequality in~\eqref{eq:limsup_bound_1}, and introducing    
\begin{align}
\label{eq:e2_def}
\begin{split}
  e_2^t  \triangleq  (1+\widehat{L})
  \smallsum_{i=1}^N \!
  \smallsum_{\ell=1}^{B}  
  \Big( 
  \Big \| \bx_{(i,:)}^{\theta_t+s_i^{\theta_t}(\ell)} \!-\!
  \avgdec^{\theta_t+s_i^{\theta_t}(\ell)} 
  \Big \|
  \\
  +
  \Big \|\avgdec^{t+s_i^{t}(\ell)} \!-\! 
  \bx_{(i,:)}^{t+s_i^t(\ell)} \Big \|
   \Big ),
\end{split}
\end{align}
we can write:
\begin{equation}
\label{eq:limsup_bound_2}
\hspace{-0.2cm}
  \delta  
  \le
  (1+\widehat{L})
  \smallsum_{i=1}^N
  \smallsum_{\ell=1}^{B}
  \left \| \avgdec^{\theta_t+s_i^{\theta_t}(\ell)} - 
  \avgdec^{t+s_i^t(\ell)} \right\|
  + e_1^t + e_2^t.
\end{equation}

Since $\theta_t+s_i^{\theta_t}(\ell) $ is the \emph{last} time at which block $\ell$ 
has been updated by agent $i$ in $[\theta_t, \theta_t+T -1]$ and $\theta_t > t$, 
it must hold: $\theta_t+s_i^{\theta_t}(\ell) \ge t+s_i^{t}(\ell)$, 
for all $t \in \TT$.
We assume, without loss of generality, that $\theta_t+s_i^{\theta_t}(\ell) > t+s_i^{t}(\ell)$, 
for all $i\in\until{N}$ and $\ell\in\until{B}$. Hence, all the intervals 
$[t+s_i^{t}(\ell), \theta_t+s_i^{\theta_t}(\ell)] $ are nonempty.
Using~\eqref{eq:avg_dec_block_evolution} to bound
$\| \avgdec^{\theta_t + s_i^{\theta_t}(\ell)} - \avgdec^{t + s_i^t(\ell)} \|$ 
in~\eqref{eq:limsup_bound_2}, we can write
\begin{align*}
  \delta 
 \le
 &
  \,c_{10}\,
  \smallsum_{i=1}^N 
  \smallsum_{\ell=1}^{B} 
  \smallsum_{\ell'=1}^{B} 
  \smallsum_{\tau = t+s_i^{t}(\ell)}^{\theta_t+s_i^{\theta_t}(\ell)-1} 
  \smallsum_{h=1}^N 
  \gamma^\tau 
  \| \Deltax_{(h,\ell')}^\tau \|
  + e_1^t + e_2^t
  \\
  \le
  &
    \,c_{11}\,
  \smallsum_{\tau = t}^{\theta_t+T-1} 
  \smallsum_{h=1}^N 
  \smallsum_{\ell = 1}^{B}
  \gamma^\tau 
  \| \Deltax_{(h,\ell)}^\tau \|
  + e_1^t + e_2^t\\
    =
  &\,
    c_{11}\,
  \smallsum_{\tau = t}^{\theta_t+T-1} 
  \smallsum_{i=1}^N 
  \gamma^\tau 
  \| \Deltax_{(i,\ell_i^\tau)}^\tau \|
  + e_1^t + e_2^t,
\end{align*}
for some positive, finite scalars $c_{10}$ and $c_{11}$, where the last equality 
follows from the fact that, at time $t$, agent $i$ optimizes only block 
$\ell_i^t$, implying $\| \Deltax_{(i,\ell)}^t \| = 0$, for all $\ell \neq \ell_i^t$.

Note that, for each $i\in\until{N}$, $t+T -1$ is the last time agent $i$ selects 
block $\ell_i^{t+T-1}$ in the interval $[t, t+ T -1]$.
Therefore, for all $i\in\until{N}$,
\begin{align*}
  \left \|
    \Deltax_{(i,\ell_i^{t+T-1})}^{t+T-1} 
  \right\|
   =
  \left \|
    \Deltax_{(i,\ell_i^{t+T-1})}^{t+s_i^t(\ell_i^{t+T-1})} 
  \right\| 
  \stackrel{\eqref{eq:Delta^t}}{\le} 
  \bDelta^{t}.
\end{align*}
Hence, we can write 
\begin{align*}
\begin{split}
  \delta 
  &
  \le
  c_{12}
  \Big(
    \smallsum_{\tau = t}^{t + T-1} 
  \gamma^{\tau} 
  \smallsum_{i=1}^N
   \| \Deltax_{(i,\ell_i^\tau)}^\tau \|
  +
  \smallsum_{\tau = t+T}^{\theta_t + T-1} 
  \gamma^{\tau} 
  \bDelta^{\tau -T +1}
  \Big)
  \\
  & \hspace{1.5em}
  + e_1^t + e_2^t 
  \\
  & 
=
  c_{12}
  \smallsum_{\tau = t+1}^{\theta_t } 
  \gamma^{\tau + T -1} 
  \bDelta^{\tau}
  + e_1^t + e_2^t  + e_3^t,
\end{split}
\end{align*}
for some positive, finite scalar $c_{12}$, where we set 
\begin{align}
\label{eq:e3_def}
  e_3^t 
  \triangleq 
  c_{12} \smallsum_{\tau = t}^{t + T-1} 
  \gamma^{\tau} \smallsum_{i=1}^N
   \| \Deltax_{(i,\ell_i^\tau)}^\tau \|.
\end{align}
Since $\bDelta^{\tau}  > \delta$, for $\tau \in [t + 1,\theta_t]$ 
[cf.~\eqref{eq:limsup_cond}], we have
\begin{align}
\begin{split}
  \delta \!
    &
  \le
  c_{12}
  \smallsum_{\tau = t+1}^{\theta_t} \gamma^{\tau + T-1}
   \smallsum_{i=1}^N \smallsum_{\ell=1}^{B} 
  \| \Deltax_{(i,\ell)}^{\tau+s_i^\tau(\ell)} \|
+ e_1^t + e_2^t + e_3^t
  \\
      &
  \le
   \frac{c_{11}}{\delta}
  \smallsum_{\tau = t+1}^{\theta_t} \gamma^{\tau + T-1}
   \left(
   \smallsum_{i=1}^N \smallsum_{\ell=1}^{B} 
  \Big \| \Deltax_{(i,\ell)}^{\tau+s_i^\tau(\ell)} \Big \|
  \right)^{\!2}
  \!\!\!+\! e_1^t \!+\! e_2^t \!+\! e_3^t
  \\
  &
  \le
  c_{13}
  \smallsum_{\tau = t+1}^{\theta_t} \gamma^{\tau + T-1}
  \smallsum_{i=1}^N \smallsum_{\ell=1}^{B} 
  \| \Deltax_{(i,\ell)}^{\tau+s_i^\tau(\ell)} \|^2
  \!+\! e_1^t \!+\! e_2^t \!+\! e_3^t
\\  &
  \le
  c_{13}
  \smallsum_{\tau = t+1}^{\theta_t} \gamma^{\tau + T-1}
  \smallsum_{s=0}^{T-1}\smallsum_{i=1}^N \smallsum_{\ell=1}^{B} 
  \| \Deltax_{(i,\ell)}^{\tau + s} \|^2
   \!+\! e_1^t \!+\! e_2^t \!+\! e_3^t,
\end{split}
\label{eq:refer_to_contradict}
\end{align}
for some positive, finite scalar $c_{13}$.
Using now  Proposition~\ref{prop:consensus},
Lemma~\ref{lem:best_response_consistency},
and~\eqref{eq:summable_dec}, we infer   that $e_1^t$, $e_2^t$, and $e_3^t$ [defined in~\eqref{eq:e1_def}, 
\eqref{eq:e2_def}, and \eqref{eq:e3_def}, respectively] are asymptotically 
vanishing, that is,  $e_1^t, e_2^t, e_3^t \underset{t\to\infty}{\longrightarrow} 0$. 
Furthermore, since  [due to (\ref{eq:summable_dec}) and \eqref{eq:summable_dec_window}] 
\begin{align*}
  \smallsum_{\tau=0}^{\infty} 
       \gamma^{\tau +T-1}
  \smallsum_{s =0}^{T-1}
  \smallsum_{\ell=1}^B
  \smallsum_{i=1}^N 
  \| \Deltax_{(i,\ell) }^{ \tau + s} \|^2<\infty,
\end{align*}
it must be 
\begin{align*}
  \lim_{t\to\infty} \smallsum_{\tau = t+1}^{\theta_t} \gamma^{\tau + T-1}
  \smallsum_{s=0}^{T-1}
  \smallsum_{i=1}^N 
  \smallsum_{\ell=1}^{B} 
  \| \Deltax_{(i,\ell)}^{\tau + s} \|^2 =0.
\end{align*}
Therefore there must exist a sufficient large $\bar{t}\in \mathcal{T}$ such that   
\begin{align*}
  c_{13}
  \smallsum_{\tau = t+1}^{\theta_t} \gamma^{\tau + T-1}
  \smallsum_{s=0}^{T-1}
  \smallsum_{i=1}^N 
  \smallsum_{\ell=1}^{B} 
  \| \Deltax_{(i,\ell)}^{\tau + s} \|^2
  +e_1^t + e_2^t + e_3^t 
  \leq 
  \dfrac{\delta}{2}
\end{align*}
for all $t \ge \bar{t}$, which contradicts~\eqref{eq:refer_to_contradict}.
Thus, it must be $\limsup_{t\to\infty } \bDelta^{t} = 0$, and hence  
\begin{align}
  \label{eq:xtilde_vs_x}
  \lim_{t\to\infty}
  \smallsum_{i=1}^N 
  \smallsum_{\ell=1}^{B}
  \| \Deltax_{(i,\ell)}^{t+s_i^t(\ell)}  \|
  = 0.
\end{align}
\noindent\textbf{Step 2 --  Every limit point of $\{\bar{\bs}^t\}_{t\geq 0}$ is stationary for \eqref{eq:problem}}. 
Let $\bar{\bs}^\infty$ be a limit point of $\{\bar{\bs}^t\}_{t\geq 0}$; note that such a point exists, because 
$\{\bar{\bs}^t\}_{t\geq 0}$ is bounded (cf.~Section~\ref{sec:convergence_s_bar}). By 
Lemma~\ref{lem:loc_best_response_property}, $\bar{\bs}^\infty$ is a stationary solution of 
problem~\eqref{eq:problem}, if
\begin{align}\label{eq:vanish_dec}
  \lim_{t\to\infty} \left\| \hatx_{(i,:)} \big ( \avgdec^t \big ) - \avgdec^t \right\| = 0,
\end{align}
for all $i\in\until{N}$. To prove~\eqref{eq:vanish_dec}, we first bound 
$\| \hatx_{(i,:)} \big ( \avgdec^t \big ) - \avgdec^t\|$  as follows
\begin{align*}
 \left \| \hatx_{(i,:)} \big ( \avgdec^t \big ) - \avgdec^t \right\|
 \, & 
 {\le} \,
  \smallsum_{\ell=1}^B   \left(
 \left \| \hatx_{(i,\ell)} \big ( \avgdec^t \big ) 
  - 
  \hatx_{(i,\ell)} \big ( \avgdec^{t+s_i^t(\ell)} \big ) \right\|\right.
  \\ \begin{split}
  & \hspace{3.5em} +
 \left \| \hatx_{(i,\ell)} \big ( \avgdec^{t+s_i^t(\ell)} \big ) - \avgdec_\ell^{t+s_i^t(\ell)} \right\|\smallskip 
  \\
  & \hspace{3.5em} +
  \left.\left\|\avgdec_\ell^{t+s_i^t(\ell)} - \avgdec_\ell^t \right\| \right)
  \end{split}
  \\
  \begin{split}
  & 
  \stackrel{(a)}{\le}
  \smallsum_{\ell=1}^B  \left(
 \left \| \hatx_{(i,\ell)} \big ( \avgdec^{t+s_i^t(\ell)} \big ) - \avgdec_\ell^{t+s_i^t(\ell)} \right\|\right.
   \\
  & \hspace{3.5em} + (1+\hat{L})
 \left. \left\|\avgdec^{t+s_i^t(\ell)} - \avgdec^t \right\| \right)
  \end{split}
  \\
  \begin{split}
  & 
  {\le}
  \smallsum_{\ell=1}^B  \left(
  \left\| \hatx_{(i,\ell)} \big ( \avgdec^{t+s_i^t(\ell)} \big ) - \hatx_{(i,\ell)} \big ( \bx_{(i,:)}^{t+s_i^t(\ell)} \big )\right\| \right.\\
  & \hspace{3.5em} + \left\|  \hatx_{(i,\ell)} \big ( \bx_{(i,:)}^{t+s_i^t(\ell)} \big )- \tildex_{(i,\ell)}^{t+s_i^t(\ell)} \right\| \\
  &\hspace{3.5em} +
  \left\| \Delta \bx_{(i,\ell)}^{t+s_i^t(\ell)}
  \right \|
  \\
  & \hspace{3.5em}
  + \left\| \bx_{(i,\ell)}^{t+s_i^t(\ell)} - \avgdec_\ell^{t+s_i^t(\ell)} \right\|
  \\
  & \hspace{3.5em}
  + (1+\hat{L})\,
 \left.\left \|\avgdec^{t+s_i^t(\ell)} - \avgdec^t \right\| \right)\\
   & 
  \stackrel{(b)}{\le}
  \smallsum_{\ell=1}^B  \left(
\left\|  \hatx_{(i,\ell)} \big ( \bx_{(i,:)}^{t+s_i^t(\ell)} \big )- \tildex_{(i,\ell)}^{t+s_i^t(\ell)} \right\|\right. \\
  &\hspace{3.5em} +
  \left\| \Delta \bx_{(i,\ell)}^{t+s_i^t(\ell)}
  \right \|
  \\
    & \hspace{3.5em}
  +\,(1+\widehat{L}) \left\| \bx_{(i,:)}^{t+s_i^t(\ell)} - \avgdec^{t+s_i^t(\ell)} \right\|
  \\
  & \hspace{3.5em}
  + (1+\hat{L})\,
 \left.\left \|\avgdec^{t+s_i^t(\ell)} - \avgdec^t \right\| \right)
  \end{split}
\end{align*}
where in (a) and (b) we used
the Lipschitz continuity of $\hatx_{(i,\ell)} (\bullet)$. We show next that the four terms on the RHS of the above inequality are all asymptotically vanishing, which proves \eqref{eq:vanish_dec}.                                                                                                                                                                                                                                                                                       
Invoking Proposition~\ref{prop:consensus} [cf.~\eqref{eq:consenus_agreement}], we have
\begin{equation*}
\lim_{t \to \infty} \left\| \bx_{(i,\ell)}^{t+s_i^t(\ell)} - \avgdec_\ell^{t+s_i^t(\ell)} \right\| = 0,
\end{equation*}
for all  $\ell\in \until{B}$ and $i\in \until{N}$. By definition of $t+s_i^t(\ell)$, there exists  some $\overline{\mathcal{T}}\subseteq \mathbb{N}_+$, with $|\overline{\mathcal{T}}|=\infty$, such that  
\begin{align*}
&\lim_{t \to \infty} \left\|\hatx_{(i,\ell)} \big ( \bx_{(i,:)}^{t+s_i^t(\ell)} \big )- \tildex_{(i,\ell)}^{t+s_i^t(\ell)}\right\| 
\\
&\qquad = \lim_{\overline{\mathcal{T}}\ni t\to\infty}
   \left \| 
      \hatx_{(i,\ell_i^t)} \big( \bx_{(i,:)}^t \big) - \tildex_{(i,\ell_i^t)}^t
    \right\|\stackrel{\eqref{eq:best_response_consistency}}{=} 0, 
\end{align*}
for all  $\ell\in \until{B}$ and $i\in \until{N}$.
Using \eqref{eq:xtilde_vs_x}, we have
$\lim_{t \to \infty} \|\Delta \bx_{(i,\ell)}^{t+s_i^t(\ell)} \| = 0$,
which, together to \eqref{eq:vanishing_avgdec_diff}, yields 
\begin{align*}
  \lim_{t\to \infty} \left\| \avgdec^{t+s_i^t(\ell)} - \avgdec^t \right\| = 0,
\end{align*}
for all $\ell \!\in\! \until{B}$ and $ i\! \in\! \until{\!N}$, completing the proof.

\bibliographystyle{IEEEtran}
\bibliography{distributed_nonconvex_blocks}

\end{document}